\documentclass[a4paper, 10pt]{article}

\usepackage{graphicx}
\usepackage{amsmath}
\usepackage{url}
\usepackage{amssymb}
\usepackage{amsfonts}
\usepackage{enumerate}
\usepackage{hyperref}
\usepackage{float}
\usepackage{amsthm}
\usepackage{algorithm2e}
\usepackage{authblk}

\graphicspath{{./}}

\newcommand{\R}{\mathbb{R}}
\newcommand{\N}{\mathbb{N}}
\newcommand{\Sp}[1]{\mathrm{span}\{#1\}}
\newcommand{\inner}[2]{\ifthenelse{\equal{#2}{}}{\left\langle\cdot,\cdot\right\rangle_{#1}}{\left\langle#2\right\rangle_{#1}}}
\newcommand{\norm}[2]{\ifthenelse{\equal{#2}{}}{\left\|\cdot\right\|_{#1}}{\left\|#2\right\|_{#1}}}
\newcommand{\seminorm}[2]{\ifthenelse{\equal{#2}{}}{\left|\cdot\right|_{#1}}{\left|#2\right|_{#1}}}
\newcommand{\calh}{\mathcal H}
\newcommand{\fa}{\hbox{ for all }}
\DeclareMathOperator*{\argmin}{arg\min}
\DeclareMathOperator*{\argmax}{arg\max}
\newtheorem{theorem}{Theorem}
\newtheorem{prop}[theorem]{Proposition}
\newtheorem{cor}[theorem]{Corollary}
\newtheorem{definition}[theorem]{Definition} 
\newtheorem{rem}[theorem]{Remark}
\newtheorem{example}[theorem]{Example}

\title{Sampling based approximation of linear functionals in Reproducing Kernel Hilbert Spaces}
\author[1]{Gabriele Santin \thanks{gsantin@fbk.eu, \href{http://orcid.org/0000-0001-6959-1070}{orcid.org/0000-0001-6959-1070}}}
\author[2]{Toni Karvonen \thanks{tkarvonen@turing.ac.uk}}
\author[3]{Bernard Haasdonk \thanks{haasdonk@mathematik.uni-stuttgart.de}}
\affil[1]{Center for Information and Communication Technology, Fondazione Bruno Kessler, Italy}
\affil[2]{The Alan Turing Institute, United Kingdom}
\affil[3]{Institute for Applied Analysis and Numerical Simulation, University of Stuttgart, Germany}

\begin{document}
\maketitle

\begin{abstract}
In this paper we analyze a greedy procedure to approximate a linear functional defined in a Reproducing Kernel Hilbert Space by nodal values. This procedure 
computes a quadrature 
rule which can be applied to general functionals, including integration functionals.

For a large class of functionals, we prove convergence results for the approximation by means of uniform and greedy points which
generalize in various ways several known results. A perturbation analysis of the weights and node computation is also discussed. 

Beyond the theoretical investigations, we demonstrate numerically that our algorithm is effective in treating various integration densities, and that it is 
even very 
competitive
when compared to existing methods for Uncertainty Quantification.
\end{abstract}

\section{Introduction}\label{sec:intro}

Given a strictly positive definite kernel $K:\Omega\times\Omega\to\R$ on a bounded, measurable set $\Omega\subset\R^d$ and a linear and continuous functional 
$L\in\calh'$ 
in the dual of the associated reproducing kernel Hilbert space $\calh$, we are interested in the construction of quadrature-like formulas (or nodal 
approximants) that 
approximate 
$L(f)$ for all $f\in\calh$.
This means that we look for pairwise distinct centers $X:=\left\{x_i\right\}_{i=1}^n\subset\Omega$ and 
weights $W := \left(w_i\right)_{i=1}^n \in \R^n$ 
such that 
\begin{align*}
Q_{X, W, L}(f) := \sum_{i=1}^n w_i f(x_i)\approx L(f) \;\; \fa f\in \calh.
\end{align*}
We measure the approximation quality of $Q_{X,W,L}$ by means of the worst case error on the unit ball of $\calh$, i.e.,
\begin{align*}
e_{\calh}\left(Q_{X, W,L}\right) := \sup\limits_{\norm{\calh}{f}\leq 1} \seminorm{}{Q_{X, W, L}(f) - L(f)}.
\end{align*}
In terms of this error measure, for a given set $X$ there exist optimal weights $W^*:=W^*(X)$, i.e.,
\begin{align*}
W^*:= \argmin\limits_{W\in\R^n} e_{\calh}\left(Q_{X, W,L}\right),
\end{align*}
and we use the notation $Q_{X,L}:=Q_{X,W^*,L}$ for the weight-optimal quadrature formula. 

We will discuss in the following (see Proposition \ref{prop:optimal_weights}) how these weights can be computed explicitly, but here we anticipate in 
particular that, if $v_L\in\calh$ is the Riesz representer 
of 
$L$, 
and if $\Pi_X(v_L)$ is the $\calh$-orthogonal projection of $v_L$ into $V(X):=\Sp{K(\cdot, x)\,\colon\, x\in X}\subset\calh$ , then it holds that
\begin{align*}
\Pi_X v_L = \sum_{i=1}^n w^*_i K(\cdot, x_i)
\end{align*}
and
\begin{align*}
e_{\calh}\left(Q_{X, L}\right) = \norm{\calh}{v_L - \Pi_X v_L}.
\end{align*}
Since it is well known that $\Pi_X v_L$ coincides with the interpolant of $v_L$ on the points $X$, this means that the optimal weights can be easily computed 
via 
standard 
kernel-based interpolation, and the worst-case error coincides with the $\calh$-norm interpolation error of $v_L$ on $X$.

Assuming that these optimal weights are used, the question of selecting the centers remains open, and the goal of this paper is to analyze a 
particularly simple greedy algorithm to do so. The algorithm has been introduced in \cite{Schaback2014c}, although we give here a more explicit 
characterization. It 
starts from an empty set of centers and iteratively chooses a point among the ones that guarantee the 
maximal reduction of the worst-case error. This allows to distribute adaptive and possibly non uniform centers tailored 
to the specific $L$, and this property is particularly attractive especially in high dimensions. 

We show that the algorithm is in fact the $f/P$-greedy algorithm of \cite{Mueller2009} known in kernel interpolation, and applied to 
the Riesz representer $v_L$ of $L$. In particular we recall how it can be efficiently described and implemented in terms of the Newton basis as in 
\cite{Muller2009,Pazouki2011}.

We then prove two types of error estimates for the approximation of a special class of functionals $L\in\calh'$, namely those which are continuous on 
$\calh$ also with respect to the $L_q(\Omega)$ norm for some $1\leq q\leq\infty$, 
i.e., such that there exists $1\leq q\leq\infty$ and $c_L\geq 0$ with
\begin{align}\label{eq:smoothness_L}
|L(f)|\leq c_L \norm{L_q(\Omega)}{f}\;\;\fa f\in\calh. 
\end{align}
First, for certain translational invariant kernels we provide convergence orders for weight-optimal quadrature 
rules with quasi-uniform sets of centers. These results are a direct consequence of the error rates known for kernel interpolation, and the resulting speed of 
convergence depends on the input dimension $d$, the value of $q$, and the smoothness of the kernels. In particular, smoother kernels lead to faster 
convergence. 
For some specific functionals, which are included in our analysis, this result coincides with the ones of \cite{Kanagawa2019}.

Second, for fairly general kernels we prove convergence with rate $n^{-1/2}$ for the new greedy algorithm, where $n$ is the number of centers. Although various 
experiments suggest that 
this rate is far from optimal, it is nevertheless dimension independent and it applies to a wide class of kernels, namely, continuous kernels on bounded 
domains. Moreover, also for translational 
invariant kernels this rate is strictly better than the one for uniform points for a 
range of values of $d$, $q$, and of the smoothness of the kernel, where the range is wider for increasing $d$ if $q$ is fixed, i.e., the result improves with 
the growth of the input dimension.

The motivation for the study of the class of functionals \eqref{eq:smoothness_L} comes from integration functionals $L(f):=\int_{\Omega} f(x) \nu(x) dx$, where 
$\nu\in L_{p}(\Omega)$ for some $1\leq p\leq \infty$. Indeed, in this case we can take $q$ such that 
$\frac1p+\frac1q=1$ and $c_L:= \norm{L_p(\Omega)}{\nu}$, since
\begin{align*}
\left|L(f)\right| \leq \int_{\Omega} \left|f(x)\right| \left|\nu(x)\right| dx \leq \norm{L_p(\Omega)}{\nu} \norm{L_q(\Omega)}{f}.
\end{align*}
In this case $Q_{X,L}$ is a quadrature formula in the classical sense. Observe in particular that this class includes also the case of 
$\nu\in L_p(\Omega)\setminus L_{\infty}(\Omega)$ for some $1\leq p<\infty$, which is not covered in \cite{Kanagawa2019}.

However, our analysis comprises other interesting examples that will be discussed in Section \ref{sec:approx}. In particular, we can consider functionals $L$ 
which 
represent any quadrature rule with bounded weights, including Monte Carlo ones. In this case, constructing a quadrature rule $Q_{X,L}$ via the greedy algorithm 
means to 
find an approximated quadrature with possibly much less centers, or a compression of the quadrature $L$.

Moreover, given the equivalence between the weight-optimal quadrature of $L$ and the interpolation of $v_L$, and the equivalence between the new greedy 
algorithm for 
quadrature and the $f/P$-greedy algorithm for interpolation, our analysis alternatively applies to the interpolation via $f/P$-greedy of the class of functions 
that are 
Riesz representers of functionals satisfying the condition \eqref{eq:smoothness_L}. The convergence results will then be also be convergence results in $\calh$ 
for 
interpolation. 
These results are potentially very interesting, as we remark that for general functions $v\in \calh$, the error $\norm{\calh}{v - \Pi_X v}$ can decay 
arbitrarily slowly 
even for nicely chosen $X$ (see e.g. \cite[Section 8.4.2]{Iske2019}). Instead, for this class of functions, rates of convergence are obtained here for 
interpolation with 
both uniform and greedy points. 

Two notable examples comprised in this function class are worth mentioning. First, using $q=\infty$, for any $B>0$ the class contains the set 
\begin{align*}
\calh_B:=\left\{v:= \sum_{i\in I} \alpha_i K(\cdot, x_i) \,:\, \sum_{i\in I}|\alpha_i|\leq B\right\}\subset\calh,
\end{align*}
which is commonly used to study convergence rates of greedy algorithms (see e.g. \cite{DeVore1996,Temlyakov2008,Barron2008}). For this set our results 
on the greedy quadrature coincide with the rates obtained in \cite{Wirtz2013} for the $f/P$-greedy algorithm. 

Second, for Mercer kernels (e.g., continuous on a bounded $\Omega$, see \cite{Steinwart2012} for a general analysis) it can be proven that the operator 
$
 T: L_2(\Omega)\to L_2(\Omega)
$
given by
\begin{align*}
T(f) := \int_{\Omega} K(x, y) f(y) dy,
\end{align*}
has an image $I_T:= T(L_2(\Omega))$ which  is dense in $\calh$. Functions in $I_T$ are central in the study of superconvergence in kernel spaces 
\cite{Schaback1999a,Schaback2018a}, and they are covered here with $q=2$.

On the computational side we remark that both the selection of the greedy points and the computation of the optimal weights can be performed by the sole 
knowledge of 
the Riesz representer $v_L$. As we will recall, this can be computed rather efficiently and explicitly and, when it needs instead to be approximated, we give 
stability 
bounds on the resulting perturbed quadrature formula. This easiness of computation is very advantageous if the 
quadrature rule 
is then applied to the approximation of $L(f)$ for a given $f$ whose evaluation is expensive. This is the case for example for Uncertainty 
Quantification, where an integration 
functional is used to estimate the mean and variance of $f$.
Alternatively, when only few evaluations are available it can be advantageous to leverage the well-known equivalence between kernel interpolation and Gaussian 
process regression~\cite[Chapter~17]{FasshauerMcCourt2015} and view $Q_{X,L}(f)$ as a Gaussian random variable whose standard deviation, equal to the worst 
case 
error, attempts to quantify the epistemic uncertainty in the approximation $Q_{X,L}(f) \approx L(f)$.
This approach has been especially popular in integration, where it is known as the Bayesian quadrature~\cite{Larkin1972,OHagan1991,Briol2019}.
The greedy algorithm we study here is sometimes called the sequential Bayesian quadrature in this context~\cite{HuszarDuvenaud2012}.

We mention also that other data-based algorithms to approximate linear functionals exists in different settings, and are actively investigated e.g. in the 
setting of empirical interpolation and reduced order modelling \cite{Brown_2016,Yano2019,10.1093/imanum/drx072,Antil2013}.

The paper is structured as follows. In Section \ref{sec:approx} we recall some basic facts on kernel spaces, list some properties of linear and continuous 
functionals on 
these spaces, and recall the computation and properties of weight-optimal quadrature rules and their connection with interpolation. Section \ref{sec:greedy} 
introduces 
the greedy algorithm and discusses its equivalence with the $f/P$-greedy algorithm for interpolation. The convergence results for uniform points and 
translational 
invariant kernels are discussed in Section \ref{sec:convergence-uniform}, while the ones for greedy points and general kernels are presented in Section
\ref{sec:convergence-greedy}. Some stability results are shown in Section \ref{sec:stability}, providing in particular bounds on the worst case error of a 
quadrature rule obtained from a perturbed Riesz representer. Finally, the greedy method is tested on both synthetic examples and on a benchmark problem in 
Uncertainty Quantification in Section \ref{sec:numerics}.

\section{Kernels and approximation of linear functionals}\label{sec:approx}

We recall some basics of kernel theory, and for a more general treatment we refer e.g. to \cite{Wendland2005,Fasshauer2007}.
A strictly positive definite (s.p.d.) kernel on $\Omega\subset\R^d$ is a symmetric function $K \colon 
\Omega\times\Omega\to\R$ such that for 
all $n\in\N$ and for all sets 
$X_n:=\{x_i\}_{i=1}^n\subset\Omega$ of pairwise distinct 
points, the kernel matrix 
$A:=A(K, X_n)\in \R^{n\times n}$ defined by $A_{ij}:= K(x_i, x_j)$ is positive definite. We assume this condition in the following, and additionally that 
$\Omega\subset\R^d$ is bounded and Lebesgue measurable, and $K$ is continuous in both variables.

Each s.p.d.\ kernel is uniquely associated to a reproducing kernel Hilbert space (RKHS) $\calh:=\calh_K(\Omega)$ with inner product $\inner{\calh}{}$, which is 
usually 
called native space of $K$ on $\Omega$, and which is a Hilbert space of functions $f:\Omega \to\R$ such that $K(\cdot, x)\in\calh$ for all $x\in\Omega$ and 
$\inner{\calh}{f, K(\cdot, x)} = f(x)$ for all $x\in\Omega$ and $f\in\calh$. 

These two properties mean in particular that the function $v_{\delta_x}(\cdot):= K(\cdot, x)$ is an element of $\calh$ for all $x\in\Omega$, and that it is the 
Riesz representer of the linear functional $\delta_x:\calh\to\R$, $\delta_x(f):= f(x)$, which is thus continuous, i.e.,  $\delta_x\in\calh'$. Actually also the 
converse holds, 
i.e., any Hilbert space where the set $\left\{\delta_x:x\in\Omega\right\}$ is contained in $\calh'$ is an RKHS, and the corresponding kernel is strictly 
positive definite 
whenever the elements of this set are also linearly independent. 

Observe that for these particular functionals the Riesz representer $v_{\delta_x}$ can be obtained by applying the functional $\delta_x$ to one of the two 
variables of the kernel, i.e., $v_{\delta_x} = K(\cdot, x) = \left(\delta_x\right)^y (K(\cdot, y))$, where the upper index denotes the variable with respect to 
which the functional is applied. This is actually always the case, as we recall in the next proposition.
\begin{prop}[Riesz representer {\cite[Theorem 16.7]{Wendland2005}}]
Let $L\in\calh'$ be a linear and continuous functional on $\calh$. Then the Riesz representer $v_L\in\calh$ of $L$ is given by $v_L(\cdot) = L^y(K(\cdot,y))$. 
\end{prop}

Strictly positive definite kernels allow especially to solve 
interpolation problems, as stated in the following proposition, which is a collection of various classical results (see e.g. \cite{Wendland2005})

\begin{prop}[Kernel interpolation]\label{prop:interpolation}
For an s.p.d. kernel $K$ and a set $X:= \{x_i\}_{i=1}^n\subset\Omega$ of pairwise distinct points, we denote as $V(X):=\Sp{K(\cdot, x)\,\colon\, x\in X}$ the 
subspace of $\calh$ spanned by the kernel translates at $X$, and as $\Pi_{X}: \calh \to V(X)$ the $\calh$-orthogonal projector into $V(X)$.

Then for each $v\in\calh$ the function $\Pi_{X} v$ interpolates $v$ at $X$. Moreover, we have 
\begin{align}\label{eq:int_def}
\Pi_{X} v = \sum_{i=1}^n \alpha_i K(\cdot, x_i),
\end{align}
where $\alpha\in\R^n$ is the unique solution of the linear system $A \alpha = \left(v(x_1), \dots, v(x_n)\right)^T$ with $A$ the kernel matrix of $K$ on $X$.
\end{prop}

This proposition implies in particular that the interpolant of an arbitrary function in $\calh$ can be computed on arbitrary pairwise distinct points. 
Moreover, this 
interpolant coincides with the orthogonal projection into $V(X)$, and it is thus an optimal approximant with respect to the $\calh$-norm. As mentioned in the 
introduction, if these properties are translated to the approximation of a linear functional the following results are derived. 

\begin{prop}[Weights-optimal quadrature]\label{prop:optimal_weights}
Let $L\in\calh'$ and let $X:=\{x_i\}_{i=1}^n\subset\Omega$ be pairwise distinct. For a given set of weights $W:=\left(w_i\right)_{i=1}^n\in\R^n$ define
\begin{align*}
Q_{X, W, L}(f): = \sum_{i=1}^n w_i f(x_i),\;\; f\in \calh,
\end{align*}
and the corresponding worst-case error
\begin{align*}
e_{\calh}\left(Q_{X, W,L}\right) := \sup\limits_{\norm{\calh}{f}\leq 1} \seminorm{}{Q_{X, W, L}(f) - L(f)}. 
\end{align*}
Then there exist unique weights that minimize the worst case error given $L$ and $X$, i.e., 
\begin{align}\label{eq:optimal_weights}
W^*:= \argmin\limits_{W\in\R^n} e_{\calh}\left(Q_{X, W,L}\right),
\end{align}
and they are the coefficients of the orthogonal projection of $v_L$ into $V(X)$, i.e.,  
\begin{align}\label{eq:proj_vs_weights}
\Pi_X v_L = \sum_{i=1}^n w^*_i K(\cdot, x_i).
\end{align}
Moreover, $\Pi_X v_L$ is the Riesz representer of the functional $Q_{X, L} := Q_{X, W^*, L}\in\calh'$, it holds
\begin{align}\label{eq:int_vs_quad}
Q_{X, L}(f) = L\left(\Pi_X f\right) \;\;\fa\;\; f\in\calh,
\end{align}
and 
\begin{align}\label{eq:int_vs_quad_error}
e_{\calh}\left(Q_{X, L}\right) = \norm{\calh}{v_L - \Pi_X v_L}. 
\end{align}
\end{prop}
\begin{proof}
First observe that for any given $W\in\R^n$ the quadrature formula as an operator $Q_{X, W, L}:\calh\to \R$ is clearly linear. It is also continuous since  
for any $f\in\calh$ it holds
\begin{align*}
\seminorm{}{Q_{X,W,L}(f)} &= \seminorm{}{\sum_{i=1}^n w_i f(x_i)} 
= \left|\sum_{i=1}^n w_i \inner{\calh}{K(\cdot, x_i), f}\right|
= \left|\inner{\calh}{\sum_{i=1}^n w_i K(\cdot, x_i), f}\right|\\
&\leq \norm{\calh}{f} \norm{\calh}{\sum_{i=1}^n w_i K(\cdot, x_i)}
= \norm{\calh}{f} \sqrt{\sum_{i=1}^n \sum_{j=1}^n w_i w_j K(x_j, x_i)},
\end{align*}
and this proves both that $Q_{X,W,L}\in\calh'$ and that its Riesz representer is 
\begin{align}\label{eq:weights-intermediate1}
v_W :=  \sum_{i=1}^n w_i K(\cdot, x_i). 
\end{align}
Using these facts, and still for generic weights $W\in\R^n$, we can bound the worst case error as 
\begin{align}\label{eq:weights-intermediate2}
\nonumber e_{\calh}\left(Q_{X, W,L}\right) 
&= \sup\limits_{\norm{\calh}{f}\leq 1} \seminorm{}{Q_{X, W, L}(f) - L(f)}
= \sup\limits_{\norm{\calh}{f}\leq 1} \left|\inner{\calh}{v_W, f} - \inner{\calh}{v_L, f}\right|\\
\nonumber&= \sup\limits_{\norm{\calh}{f}\leq 1} \left|\inner{\calh}{v_W - v_L, f}\right|
\leq \sup\limits_{\norm{\calh}{f}\leq 1} \norm{\calh}{v_W - v_L} \norm{\calh}{f}\\
&= \norm{\calh}{v_W - v_L},
\end{align}
and since equality is reached for $f:= (v_W - v_L)/ \norm{\calh}{v_W - v_L}$ with $\norm{\calh}{f}=1$, we can conclude that $e_{\calh}\left(Q_{X, W,L}\right)  
= 
\norm{\calh}{v_W - v_L}$.

Now, \eqref{eq:proj_vs_weights} holds since for any choice of $W$ we have from \eqref{eq:weights-intermediate1} that $v_W\in V(X)$ and thus 
$e_{\calh}\left(Q_{X, W,L}\right) = \norm{\calh}{v_W - v_L}$ is minimized uniquely by $v_W:= \Pi_{X} \left(v_L\right)$, by the best approximation property of 
orthogonal projections. We thus have that 
\eqref{eq:weights-intermediate2} becomes \eqref{eq:int_vs_quad_error}, and by uniqueness of the projection also \eqref{eq:optimal_weights} holds.

Finally, we just use the fact that orthogonal projections are self adjoint to obtain
\begin{align*}
Q_{X, L}(f) = \inner{\calh}{\Pi_X v_L, f}=\inner{\calh}{v_L, \Pi_X f} = L\left(\Pi_X f\right) \;\;\fa f\in\calh,
\end{align*}
which proves \eqref{eq:int_vs_quad}.
\end{proof}
Observe that this proposition implies that in practice the quadrature weights can be found just by computing the interpolant of $v_L$ at the 
points $X$, and, according to Proposition \ref{prop:interpolation}, this corresponds to the solution of a linear system. Moreover, \eqref{eq:int_vs_quad} 
proves that 
applying the quadrature formula is equivalent to exactly applying $L$ to the interpolant, or that $Q_{X, L}$ is exact on $V(X)$.

\begin{rem}[Positive definite kernels]
The results of this section can be formulated also for positive definite kernels, i.e., those for which the kernel matrix is required to be only positive 
semidefinite. 

However, in this case some complications arise since the kernel matrix can be singular also for pairwise distinct points $X$, and in particular the elements 
$K(\cdot, x_i)$ do not need to be linearly independent, and thus they span $V(X)$ without being a basis. Nevertheless, the same results can be derived if more 
attention is paid, for 
example  by showing that different representations \eqref{eq:int_def} can describe a unique function. 

More importantly, we do not explicitly extend the current presentation to positive definite kernels since it is not clear if the greedy algorithm of the next 
section, 
which is the main topic of this paper, can be run without producing singular matrices (and thus an early termination) in the case of (non strictly) positive 
definite kernels.
\end{rem}

Although the construction of $Q_{X, L}$ and the greedy algorithm that we will introduce work for any $L\in\calh'$, we recall that our error analysis will apply 
only to 
functionals such that there exists $1\leq q\leq\infty$ and $c_L\geq 0$ with
\begin{align}\label{eq:smoothness_L_2}
|L(f)|\leq c_L \norm{L_q(\Omega)}{f}\;\;\fa f\in\calh. 
\end{align}
Observe that this definition is well posed, since $\calh\subset L_q(\Omega)$ for all $1\leq q\leq\infty$ and for all $f\in\calh$ since $\Omega$ is assumed to 
be bounded 
and $K$ continuous.

Now that the equivalence between worst-case quadrature with optimal weights and interpolation has been detailed, we conclude this section with a more precise
discussion of some 
relevant examples of functionals satisfying the condition \eqref{eq:smoothness_L_2}. Some of the following examples have been already addressed in Section 
\ref{sec:intro}.

\begin{example}\label{example:measure}
If $L(f):=\int_{\Omega} f(x) \nu(x) dx$ with $\nu\in L_{p}(\Omega)$ for some $1\leq p\leq \infty$, we can take $q$ such that 
$\frac1p+\frac1q=1$ and we have $c_L:= \norm{L_p(\Omega)}{\nu}$ since
\begin{align*}
\left|L(f)\right| \leq \int_{\Omega} \left|f(x)\right| \left|\nu(x)\right| dx \leq \norm{L_p(\Omega)}{\nu} \norm{L_q(\Omega)}{f}.
\end{align*}
In this case $Q_{X,L}$ is a classical quadrature rule.
\end{example}

\begin{example}\label{example:discrete_measure}
If $L(f):= \sum_{i\in I} \rho_i f(z_i)$ where $I\subset \N$ is a countable index set, $\left\{\rho_i\right\}_{i\in I}\subset\R$, 
$\left\{z_i\right\}_{i\in I}\subset\Omega$, and if 
$\rho:=\left\{\rho_i\right\}_{i\in 
I}\in\ell_1(I)$, then we can take $q:=\infty$ and $c_L:=\norm{\ell_1(I)}{\rho}=\sum_{i\in I} \left|\rho_i\right|$ since
\begin{align*}
\left|L(f)\right| \leq \sum_{i\in I} \left|\rho_i\right| \left|f(z_i)\right| \leq \left(\max\limits_{i\in I} \left|f(z_i)\right|\right) \norm{\ell_1(I)}{\rho}  
\leq  
\norm{L_{\infty}(\Omega)}{f} \norm{\ell_1(I)}{\rho}.
\end{align*} 
This includes for example any  quadrature formula with weights $\rho:=\left\{\rho_i\right\}_{i\in 
I}\in\ell_1(I)$ and nodes $\left\{z_i\right\}_{i\in I}\subset\Omega$.
This is the case for example of quadrature rules with positive and bounded weights, and in particular of any 
Monte Carlo quadrature with $M: = |I|$, since $\sum_{i\in I} \rho_i = \sum_{i=1}^M {|\Omega|}/{M} = |\Omega|$.

In this case $Q_{X,L}$ can be understood as a compression of the quadrature rule, if $|X|\leq |I|$. Or, even for $|X|= |I|$, $Q_{X,L}$ is weight-optimal
and thus can provide a strictly better worst-case error than $L$.
\end{example}

Considering instead functions which are Riesz representers of functionals satisfying \eqref{eq:smoothness_L_2}, we have the following examples.

\begin{example}\label{example:temlyakov_class}
Given a number $B>0$, the functions in the class 
\begin{align*}
\calh_B:=\left\{v:= \sum_{i\in I} \alpha_i K(\cdot, x_i) : \sum_{i\in I}|\alpha_i|\leq B\right\}
\end{align*}
are the Riesz representers of functionals of Example \ref{example:discrete_measure} with $ \norm{\ell_1(I)}{\rho} \leq B$. This set is commonly used to study 
convergence 
rates of greedy algorithms (see e.g. \cite{DeVore1996,Temlyakov2008,Barron2008}), and for this set our results 
on the greedy algorithm coincide with the rates obtained in \cite{Wirtz2013}
\end{example}

\begin{example}\label{example:mercer}
Under the present assumptions the kernel is a Mercer kernel, and  it can be proven (see e.g. \cite[Chapter 10]{Wendland2005}) that the operator 
$
 T: L_2(\Omega)\to L_2(\Omega)
$
given by
\begin{align*}
T(f) := \int_{\Omega} K(x, y) f(y) dy,
\end{align*}
is compact and self adjoint. It has a sequence $\{\lambda_j\}_{j\in\N}$ of non increasing and positive eigenvalues and corresponding 
$L_2(\Omega)$-orthonormal eigenvectors $\{\varphi_j\}_{j\in\N}$ such that $\{\lambda_j^{1/2} \varphi_j\}_{j\in\N}$ is an 
$\calh$-orthonormal basis of $\calh$. The image 
$I_T:= T(L_2(\Omega))$ is dense in $\calh$, and every function $v\in I_T$ is the Riesz representer of a functional $L$ which satisfies 
\eqref{eq:smoothness_L_2}. Indeed, 
if $u\in L_2(\Omega)$ is such that $T(u) = v$, it can be proven that for all $f\in\calh$ it holds
$
\inner{\calh}{f, T(u)} = \inner{L_2(\Omega)}{f, u},
$
and thus
\begin{align*}
\left|L(f)\right| :&= \left|\inner{\calh}{f, v}\right| = \left|\inner{\calh}{f, T(u)}\right| = \left|\inner{L_2(\Omega)}{f, u}\right| \\
&\leq \norm{L_2(\Omega)}{u} \norm{L_2(\Omega)}{f}\;\;\fa\;\; f\in\calh.
\end{align*}
It follows that \eqref{eq:smoothness_L_2} holds with $c_L:= \norm{L_2(\Omega)}{u}$ and $q:= 2$. 

These functions are the easiest example of the class analyzed in \cite{Schaback1999a,Schaback2018a} to study superconvergence phenomena in $\calh$, i.e., 
functions for 
which kernel interpolation with uniform points leads to an improved convergence order. 

\end{example}

We conclude this general section by stating in the following proposition an obvious fact that will be useful later.
\begin{prop}[Restriction of $L$]\label{prop:restriction}
Assume that $L\in\calh'$, and assume that there exists $1\leq q\leq \infty$ such that $L$ is continuous w.r.t. the $L_q$-norm on $L_q(\Omega)\cap \calh$,  with 
norm bounded by $c_L$. Then for 
any 
$\calh$-closed subspace $V\subset\calh$ also the functional $L_V:= L\circ \Pi_V$ is continuous on $L_q(\Omega)\cap \calh$, with norm $c_{L_V}\leq c_L$.
\end{prop}
\begin{proof}

Since $L$ is continuous from $L_q(\Omega)\cap \calh$ to $\R$ with constant $c_L$, i.e.,
\begin{align*}
\sup\limits_{f\in\calh, f\neq 0} \frac{|L(f)|}{\norm{L_q(\Omega)}{f}}=c_L,
\end{align*}
then
\begin{align*}
c_{L_V}
&:= \sup\limits_{f\in\calh, f\neq 0} \frac{|L_{V}(f)|}{\norm{L_q(\Omega)}{f}}
= \sup\limits_{f\in\calh, f\neq 0} \frac{|L(\Pi_V(f))|}{\norm{L_q(\Omega)}{f}}\\
&= \sup\limits_{f\in V, f\neq 0} \frac{|L(f)|}{\norm{L_q(\Omega)}{f}}
\leq \sup\limits_{f\in \calh, f\neq 0} \frac{|L(f)|}{\norm{L_q(\Omega)}{f}}\\
&=c_L,
\end{align*}
which is the desired bound.
\end{proof}

\section{Convergence for uniform points and translational invariant kernels}\label{sec:convergence-uniform}
We first analyze convergence rates for quadrature formulas that use uniform points on spaces generated by translational invariant kernels.
In this section we thus assume that $K(x, y) := \phi(x-y)$ for some $\phi:\R^d\to\R$, and that $\phi$ has a generalized Fourier transform $\hat \phi$ such 
that there 
exists $\tau>0$, and $c, C>0$ with 
\begin{align}\label{eq:fourier_exponent}
c \left(1 + \norm{2}{\omega}^2\right)^{-\tau} \leq \hat\phi(\omega)\leq C \left(1 + \norm{2}{\omega}^2\right)^{-\tau} \;\;\fa\;\; \omega\in\R^d.
\end{align}
If additionally $\Omega$ has a Lipschitz boundary and it satisfies an interior cone condition, then $\calh$ is norm equivalent to the Sobolev space 
$W_2^{\tau}(\Omega)$, and in particular there exists a constant $c_E>0$ such that 
\begin{align}\label{eq:norm_equivalence}
\norm{W_2^{\tau}(\Omega)}{u}\leq c_E \norm{\calh}{u}\;\;\fa\;\; u\in\calh. 
\end{align}
Observe that this norm equivalence is possible only if $\tau>d/2$, since it implies in particular that $W_2^{\tau}(\Omega)$ is an RKHS (see e.g. Chapter 10 in
\cite{Wendland2005}).

In this case we can use the following sampling inequality from \cite{WendlandRieger2005}.
Here and in the following we denote $(x)_+:=\max(x, 0)$ for $x\in\R$ and, for $u:\Omega\to\R$, $\norm{\ell_{\infty}(X)}{u}$ denotes the maximum absolute value 
of $u$ evaluated on $X\subset\Omega$. Moreover, we use the fill distance
\begin{align*}
h_X: = \sup\limits_{x\in\Omega}\min\limits_{y\in X} \norm{2}{x-y}
\end{align*}
and the separation distance
\begin{align*}
q_X:=\frac12 \min\limits_{x_i\neq x_j, x_i,x_j\in X} \norm{2}{x_i - x_j} 
\end{align*}
to quantify the distribution of the points $X$ in $\Omega$.

\begin{theorem}[Sampling inequality \cite{WendlandRieger2005}]\label{th:sampling_inequality}
Let $\Omega\subset\R^d$ be bounded and satisfy an interior cone condition. 
Let $1\leq q \leq\infty$ and $\tau\in\R$ be such that $\tau>d/2$. Then there exist
$c_S>0$ and $h_0>0$ such that if $X\subset\Omega$ is finite and $h_X\leq h_0$, then for all $u\in W_2^{\tau}(\Omega)$ it holds
\begin{align}\label{eq:sampling_Lq}
\norm{L_q(\Omega)}{u} \leq c_S \left(h_X^{\tau -d (1/2 - 1/q)_+} \seminorm{W_2^{\tau}(\Omega)}{u} + \norm{\ell_{\infty}(X)}{u}\right).
\end{align}
\end{theorem}

For general point sets, a geometric constraint implies that there exists a constant $c>0$ depending only on $\Omega$ such that $h_{X}\geq c n^{-1/d}$. 
If one uses a quasi-uniform sequence $\{X_n\}_{n\in\N}\subset\Omega$, 
$|X_n|=n$, of points, i.e., such that there exists a constant $c>0$ with $h_{X_n}\leq c q_{X_n}$ for all  $n\in\N$, then it can be proven that there 
exists a second constant $c_Q$ independent of $n$, and an index $n_0\in\N$ such that for all $n\geq n_0$ it holds
\begin{align*}
h_{X_n}\leq c_Q n^{-1/d}.
\end{align*}
We refer for example to Chapter 2 in \cite{Mueller2009} for explicit estimates of these constants. 
Using a sequence of quasi-uniform sets allows to rewrite any bound expressed in terms of $h_X$ as a bound involving only the number of points $n=|X|$.

With these tools in hand we can prove the following result.
The proof follows a standard procedure used in combination with a sampling inequality to derive an error bound, with the only difference that the 
$L_q$-continuity of $L$ 
will guarantee that the error bound is in the $\calh$-norm. We remark that the same idea has been used in \cite{Bezhaev1991, Kanagawa2019, Kanagawa2016}, as 
well as most other work on error estimates for kernel and Bayesian quadrature rules, to derive error 
bounds for the integration functionals of Example \ref{example:measure} with $\nu\in L_{\infty}(\Omega)$.

\begin{theorem}[Convergence rates for uniform points]\label{th:uniform_quad}
Under the assumptions of Theorem \ref{th:sampling_inequality}, let $L\in\calh'$ be a linear functional such that there exist $1\leq q\leq\infty$ and $c_L\geq 
0$ such that
\begin{align*}
|L(f)|\leq c_L \norm{L_q(\Omega)}{f}\;\;\fa\;\; f\in\calh, 
\end{align*}
and let $v_L\in\calh$ be its Riesz representer.

Then, if $X\subset\Omega$ is a set of pairwise distinct points with $h_X\leq h_0$, it holds
\begin{align}\label{eq:bound_norm_square}
e_{\calh}\left(Q_{X, L}\right) = \norm{\calh}{v_L - \Pi_{X} v_L } &\leq c_S c_E c_L h_X^{\tau-d (1/2 - 1/q)_+}.
\end{align}
In particular, if $\{X_n\}_{n\in\N}\subset\Omega$ is a sequence of sets of pairwise distinct points for which there exists $c_Q>0$ such that
$h_{X_n}\leq c_Q n^{-1/d}$ for all $n\in\N$, then for all $n\in\N$ with $n\geq n_0:=\left(c_Q/h_0\right)^d$ it holds 
\begin{align}\label{eq:error_uniform_n}
\norm{\calh}{v_L - \Pi_{X_n} v_L}&\leq c_U \ n^{-\frac{\tau}{d} +\left(\frac12 - \frac1{q}\right)_+},
\end{align}
where the constant $c_U:= c_L c_S c_E c_Q^{\tau-d \left(1/2 - 1/q\right)_+}$ depends on $L$ only via $q$ and $c_L$.
\end{theorem}
\begin{proof}
For any $f\in\calh$, equation \eqref{eq:int_vs_quad} and the continuity of $L$ on $L_q(\Omega)$ give
\begin{align*}
\seminorm{}{L(f) - Q_{X,L}(f)} 
&= \seminorm{}{L(f) - L(\Pi_{X}f)}
= \seminorm{}{L(f - \Pi_{X}f)}\\
&\leq c_L \norm{L_q(\Omega)}{f - \Pi_{X}f}.
\end{align*}
Now, since $u:= f - \Pi_{X}f$ vanishes on $X$, the sampling inequality of Theorem~\ref{th:sampling_inequality} and the norm equivalence 
\eqref{eq:norm_equivalence} give
\begin{align*}
\seminorm{}{L(f) - Q_{X,L}(f)} 
&\leq c_L \norm{L_q(\Omega)}{f - \Pi_X f}\\ 
&\leq c_L c_S h_X^{\tau-d (1/2 - 1/q)_+} \seminorm{W_2^{\tau}(\Omega)}{f - \Pi_X f}\\ 
&\leq c_L c_S h_X^{\tau-d (1/2 - 1/q)_+} \norm{W_2^{\tau}(\Omega)}{f - \Pi_X f}\\ 
&\leq c_L c_S c_E\ h_X^{\tau-d (1/2 - 1/q)_+} \norm{\calh}{f - \Pi_X f}\\
&\leq c_L c_S c_E\ h_X^{\tau-d (1/2 - 1/q)_+} \norm{\calh}{f},
\end{align*}
and it follows from \eqref{eq:int_vs_quad_error} that
\begin{align*}
\norm{\calh}{v_L - \Pi_{X}v_L} &= \sup\limits_{\norm{\calh}{f}\leq 1} \seminorm{}{L(f) - Q_{X,L}(f)}\leq c_L c_S c_E\ h_X^{\tau-d (1/2 - 1/q)_+}.
\end{align*}
To obtain \eqref{eq:error_uniform_n} we can just bound $h_X$ from above with $c_Q n^{-1/d}$ and obtain that for all $h_X \leq h_0$, i.e., for all $n \geq (c_Q 
/ h_0)^d$, 
holds
\begin{align*}
\norm{\calh}{v_L - \Pi_{X}v_L} &\leq c_L c_S c_E\ c_Q^{\tau-d (1/2 - 1/q)_+} n^{-\tau/d+ (1/2 - 1/q)_+}.
\end{align*}
\end{proof}

Observe that the proof is just a consequence of the $L_q$-continuity of $L$ and of the fact that the quadrature rule is an exact application of the 
functional to the interpolant (see \eqref{eq:int_vs_quad}). It is clear that similar results can be obtained for any other kernel for which error 
estimates in the $L_q$-norm are available for interpolation (and we give an example at the end of this section). Moreover, the 
rate of convergence just comes from the fact that uniform points give a good error for interpolation with 
translational invariant kernels, and in particular the bound makes no distinction between different functionals. 

It remains open to investigate if 
approximation by $L$-adapted points can achieve a better approximation rate. We expect this to be the case, and also that these rates can be achieved by 
adaptive quadratures via greedy algorithms. An example supporting this claim is discussed in Section \ref{sec:exp_singular}.

\begin{rem}[Superconvergence]\label{rem:super}
As a consequence of the theorem, for some class of functions we can derive superconvergence with respect to $L_r$-norms in the sense of 
\cite{Schaback1999a,Schaback2018a}, i.e.,a rate of convergence of kernel interpolation which is better than the one for generic functions in 
$\calh$.

Namely, taking $1\leq r\leq \infty$, for the interpolation of a generic function $u\in\calh$ the application of the inequality of Theorem 
\ref{th:sampling_inequality} which gives the standard error estimate
\begin{align}\label{eq:remark_super}
\norm{L_{r}(\Omega)}{u - \Pi_X u} 
&\leq C h_X^{\tau-d\left(\frac12-\frac1r\right)_+}\norm{\calh}{u - \Pi_X u}
\leq C h_X^{\tau-d\left(\frac12-\frac1r\right)_+}\norm{\calh}{u}.
\end{align}

On the other hand, for any function $v$ such that the functional $L:=\inner{\calh}{v, \cdot}$ satisfies the assumptions of Theorem 
\ref{th:uniform_quad} for some $q$, the theorem gives
\begin{align*}
\norm{\calh}{v - \Pi_X v} 
&\leq C h_X^{\tau-d \left(\frac12 - \frac1q\right)_+},
\end{align*}
and thus \eqref{eq:remark_super} can be improved to
\begin{align*}
\norm{L_{r}(\Omega)}{v - \Pi_X v} 
&\leq C h_X^{2\tau- d\left(\left(\frac12 - \frac1r\right)_+ + \left(\frac12 - \frac1q\right)_+\right)}.
\end{align*}
In the case of $L_2$-approximation, which is mainly addressed in \cite{Schaback1999a,Schaback2018a}, this means that this class of functions can be 
approximated with an order of $2\tau - d\left(1/2 - 1/q\right)_+$ instead of $\tau - d\left(1/2 - 1/q\right)_+$. In particular, for the class of 
functions of Example \ref{example:mercer} this result coincides with the one of \cite{Schaback1999a}, even if more general function classes are included in our 
analysis.
\end{rem}

\begin{rem}[$P$-greedy]
We recall that the same greedy algorithm for interpolation, but using the $P$-greedy selection rule (i.e., select at each iteration one of the points which 
maximize the power function in~\eqref{eq:power_function}) has been shown to produce sequences of points with $h_{X_n}\leq c n^{-\frac{1}{d}(1-\varepsilon)}$ 
for all $\varepsilon>0$ in 
\cite{SH16b}. This has been refined to hold also for 
$\varepsilon=0$ in \cite{Wenzel2019}, i.e., it actually holds that the $P$-greedy algorithm 
selects sequences of points which satisfy $h_{X_n}\leq c n^{-\frac{1}{d}}$, and it follows that quadrature based on centers selected by the $P$-greedy 
algorithm give 
exactly the approximation order of Theorem \ref{th:uniform_quad}.
\end{rem}

Finally, we mention that there are other translational invariant kernels that are not comprised in this analysis, but for which there are error statements for 
interpolation which are completely analogous to the ones of Theorem \ref{th:sampling_inequality}. 
For example, for the Gaussian and Inverse Multiquadric kernels the error estimates of \cite{Rieger2008b} allow to prove exponential rates of convergence for 
the 
approximation with sequences of uniform points. Since the proof is completely analogous to that of Theorem \ref{th:uniform_quad}, we omit it here. 
Nevertheless, we remark that a similar superconvergence as in Remark \ref{rem:super} happens also in this case, since the $\calh$-norm of the error can be 
bounded by a term decaying (exponentially) with $h_X$.

\section{The greedy algorithm}\label{sec:greedy}
We can now define and discuss the greedy algorithm. 
We assume only that $K$ is an s.p.d. kernel on a set $\Omega$ and that $L\in\calh'$ is a linear functional. We recall that the notation $Q_{X, L}$ denotes the 
fact that 
optimal weights are used, and in particular the algorithm is 
completely determined just by the selection of the set of points.

\begin{definition}[Greedy algorithm]\label{def:greedy_alg}
Let $X_0:=\emptyset$ and $Q_{X_0, L}(f):= 0$ for all $f\in\calh$. For any $n\in\N$, the greedy algorithm selects a point 
\begin{align*}
x_n \in \argmin\limits_{x\in \Omega\setminus X_{n-1}} e_{\calh}\left(Q_{X_{n-1}\cup \{x\}, L}\right)
\end{align*}
with $X_n:= X_{n-1}\cup \{x_n\}$. 
\end{definition}
Observe that thanks to Proposition \ref{prop:optimal_weights} the algorithm is equivalent to the iterative selection of points to interpolate the Riesz 
representer $v_L$ with a greedy selection rule given by
\begin{align}\label{eq:f_P_greedy}
x_n 
&\in \argmin\limits_{x\in \Omega\setminus X_{n-1}} e_{\calh}\left(Q_{X_{n-1}\cup \{x\}, L}\right) 
= \argmin\limits_{x\in \Omega\setminus X_{n-1}} \norm{\calh}{v_L - \Pi_{X_{n-1}\cup\{x\}} (v_L) }, 
\end{align}
i.e., the new point provides the locally $\calh$-optimal update of the interpolant of $v_L$. 

The locally optimal selection rule is known to be the $f/P$-greedy selection, as shown in \cite{Mueller2009,Wirtz2013}. Although well known, we prove this fact 
and also
give 
the 
complete definition of the algorithm to stress the fact that it can be efficiently implemented in an iterative way. To this end, we first recall that the 
interpolation error can be bounded similarly as the worst-case quadrature error. Indeed, in the case of interpolation it is 
common to define the power function 
\begin{align}\label{eq:power_function}
P_{X_n}(x):= \sup\limits_{\norm{\calh}{f}\leq 1} \left|f(x) - (\Pi_{X_n}f)(x)\right|,
\end{align}
and, in the language of this paper, it is clear that this is the worst-case error for the weight-optimal quadrature of the functional $L:=\delta_x$. 
In fact it holds also that
\begin{align*}
P_{X_n}(x) = \norm{\calh}{K(\cdot, x) - \Pi_{X_n}K(\cdot, x)},  
\end{align*}
since $K(\cdot, x)$ is the Riesz representer of the point-evaluation functional. 
This implies in particular that the power function is continuous in $x\in\Omega$, it vanishes if and 
only if $x\in X_n$, and $P_{\emptyset}(x)=\sqrt{K(x,x)}$ for all $x\in\Omega$.
Moreover, if $\{v_i\}_{i=1}^n$ is any $\calh$-orthonormal basis of $V(X_n)$, by the definition of orthogonal projection we clearly have
\begin{align}\label{eq:proj_Newton}
\Pi_{X_n} v= \sum_{i=1}^n \inner{\calh}{v, v_i}v_i\;\;\fa\;\; v\in\calh, 
\end{align}
and in particular
\begin{align}\label{eq:pf_Newton}
P_{X_n}(x)
&= \norm{\calh}{K(\cdot, x) - \Pi_{X_n} K(\cdot, x) } 
= \left\|K(\cdot, x) - \sum_{i=1}^n {\inner{\calh}{K(\cdot, x), v_i}} v_i\right\|_{\calh}\nonumber\\
&= \left\|K(\cdot, x) - \sum_{i=1}^n v_i(x) v_i\right\|_{\calh} = \sqrt{K(x, x) -  \sum_{i=1}^n v_i(x)^2}.
\end{align}
In the following we write $P_n$ instead of $P_{X_n}$ to simplify the notation.

Using this power function, we can now recall that the $f/P$-greedy rule \eqref{eq:f_P_greedy} selects a new point as 
\begin{align}\label{eq:f_p_greedy_selection}
x_n\in\argmax\limits_{x\in \Omega\setminus X_{n-1}} \frac{\left|v(x) -(\Pi_{X_{n-1}} v)(x)\right|}{P_{n-1}(x)},
\end{align}
and it is thus clear that the name of this selection rule is indeed given by the fact that it selects a point that maximizes the ratio between the function 
interpolation 
residual (the ``$f$'' component) and the power function (the ``$P$''  component).

To proceed and describe any iterative algorithm that works with nested sequences of interpolation (or quadrature) points, it is convenient to use the following 
Newton 
basis, which is an instance of an $\calh$-orthonormal basis. 
\begin{prop}[Newton basis \cite{Muller2009,Pazouki2011}]
Given a sequence of nested sets $\{X_n\}_{n\in\N}$ of pairwise distinct points in $\Omega$ with $X_0:= \emptyset$ and $X_{n+1}:= X_{n}\cup\{x_{n+1}\}$, the 
Newton basis is a sequence $\{v_k\}_{k\in\N}\subset\calh$ such that for all $n\in\N$ the set $\{v_k\}_{k=1}^n$ is a $\calh$-orthonormal basis of 
$V(X_n)$. This means that $V(X_n) = \Sp{v_i}_{i=1}^n$ and 
that $\inner{\calh}{v_i, v_j} = \delta_{ij}$ for all $i \neq j$ and for all $n\in\N$. Moreover, the Newton basis property $v_i(x_j) = 0$ for all $1\leq j<i\leq 
n$ is satisfied.

The basis can be constructed by a Gram-Schmidt procedure over $\left\{K(\cdot, x_i)\right\}_{i\in\N}$, which gives 
\begin{align}\label{eq:n_th_Newton}
v_1 = \frac{K(\cdot, x_1) }{P_0(x_1)}, \;\;\;\;
v_n = \frac{K(\cdot, x_n) - \sum_{k = 1}^{n-1} v_k(x_n) v_k}{P_{n-1}(x_n)}, \; n\in\N.
\end{align}

\end{prop}

Using this Newton basis it becomes easy to describe the efficient update of the interpolant for an increasing set of points, and also to prove the local 
optimality of the $f/P$-greedy selection. This result has been proven in \cite{Wirtz2013}, but we include here a more direct proof that does not make use of 
orthogonal remainders.

\begin{prop}[Efficient update and $f/P$-greedy \cite{Pazouki2011,Mueller2009}]
Let $v\in\calh$, $X_{n}:=\{x_i\}_{i=1}^n:=X_{n-1}\cup\{x_n\}\subset\Omega$, and let $\{v_k\}_{k=1}^n$ be the Newton basis of $V(X_{n})$.

Then the interpolant and the power function can be updated as
\begin{align}
\label{eq:Newton_update1}\Pi_{X_n} v &= \Pi_{X_{n-1}}v  + \inner{\calh}{v, v_n} v_n = \Pi_{X_{n-1}} v  + \inner{\calh}{v - \Pi_{X_{n-1}} v, v_n} 
v_n,\\ 
\label{eq:Newton_update2}P_{n}(x)^2 &= P_{{n-1}}(x)^2  -  v_n(x)^2. 
\end{align}
Moreover, the locally $\calh$-optimal selection rule is given by the $f/P$-selection rule, i.e.,
\begin{align}\label{eq:f_P_optimal}
\argmin\limits_{x\in \Omega\setminus X_{n-1}} \norm{\calh}{v - \Pi_{X_{n-1}\cup\{x\}} v }
& = \argmax\limits_{x\in \Omega\setminus X_{n-1}} \frac{\left|v(x) -(\Pi_{X_{n-1}} v)(x)\right|}{P_{n-1}(x)}.
\end{align}
\end{prop}
\begin{proof}
Since the Newton basis is a nested and orthonormal basis, \eqref{eq:Newton_update2} and the first equation in \eqref{eq:Newton_update1} easily follows from 
\eqref{eq:proj_Newton} and \eqref{eq:pf_Newton}. Moreover, since $\inner{\calh}{v_i, v_n} = 0$ for all $1\leq i<n$, we also have 
\begin{align*}
\inner{\calh}{v - \Pi_{X_{n-1}} v, v_n} = 
\inner{\calh}{v, v_n} - \inner{\calh}{\Pi_{X_{n-1}}v, v_n}
=\inner{\calh}{v, v_n},
\end{align*}
and this proves the second equality in \eqref{eq:Newton_update1}. 

To prove the optimality of the selection rule, consider a generic point $x\in\Omega\setminus X_{n-1}$ and assume that $\{v_k\}_{k=1}^n$ is the Newton basis of 
$V(X_{n-1}\cup\{x\})$, where $v_n$ 
corresponds to the point $x$. The update formula for the interpolant and the orthonormality of the Newton basis give
\begin{align}\label{eq:energy_splitting_0}
\nonumber\norm{\calh}{v - \Pi_{X_{n-1}\cup\{x\}} v }^2 
&= \norm{\calh}{v - \Pi_{X_{n-1}} v - \langle v, v_n \rangle_{\calh} v_n }^2\\
&= \norm{\calh}{v - \Pi_{X_{n-1}} v}^2 - \inner{\calh}{v, v_n}^2,
\end{align}
and using the definition \eqref{eq:n_th_Newton} we get
\begin{align}\label{eq:energy_splitting_1}
\inner{\calh}{v, v_n} 
&= \frac{1}{P_{n-1}(x)} \inner{\calh}{K(\cdot, x) - \sum_{k = 1}^{n-1} v_k(x) v_k, v}
= \frac{v(x) -(\Pi_{X_{n-1}} v)(x)}{P_{n-1}(x)}.
\end{align}
It follows that
\begin{align*}
\argmin\limits_{x\in \Omega\setminus X_{n-1}} \norm{\calh}{v - \Pi_{X_{n-1}\cup\{x\}} v }^2
& = \argmin\limits_{x\in \Omega\setminus X_{n-1}} \left(\norm{\calh}{v - \Pi_{X_{n-1}} v}^2 - \inner{\calh}{v, v_n}^2\right)\\
& = \argmax\limits_{x\in \Omega\setminus X_{n-1}} \inner{\calh}{v, v_n}^2\\
& = \argmax\limits_{x\in \Omega\setminus X_{n-1}} \left(\frac{v(x) -(\Pi_{X_{n-1}} v)(x)}{P_{n-1}(x)}\right)^2,
\end{align*}
which proves \eqref{eq:f_P_optimal}.
\end{proof}

This proposition and Proposition \ref{prop:optimal_weights} clearly show that the greedy algorithm of Definition \ref{def:greedy_alg} coincides with the 
$f/P$-greedy 
algorithm for the interpolation of $v_L$. In particular, the update rule \eqref{eq:Newton_update1} can be used to incrementally construct the interpolant by 
computing 
only a new Newton basis element and the corresponding coefficient each time a point is added. Moreover, this update of the interpolant and the formula 
\eqref{eq:Newton_update2} for the power function
allow the efficient update of the selection rule defined in \eqref{eq:f_P_optimal}. 

\begin{rem}[Change of basis]
Observe that, once the selection of the points is completed, it is convenient to express the interpolant back from the Newton basis to the standard basis 
$\left\{K(\cdot, x_i)\right\}_{i=1}^n$ to compute the quadrature of a function $f$. Indeed, for $\Pi_{X_n} \left(v_L\right) = \sum_{i=1}^n w_i^* K(\cdot, x_i)$ 
it holds 
\begin{align*}
Q_{X_n,L}(f) = \inner{\calh}{f, \Pi_{X_n} v_L} = \sum_{i=1}^n w_i^* \inner{\calh}{K(\cdot, x_i), v} = \sum_{i=1}^n w_i^* f(x_i),
\end{align*}
i.e., the computation of $Q_{X_n,L}(f)$ requires only the knowledge of the weights and the evaluations of $f$ on $X_n$. Using instead $\Pi_{X_n} v_L = 
\sum_{i=1}^n 
\inner{\calh}{v_L, v_i} v_i$ we have
\begin{align*}
Q_{X_n,L}(f) = \sum_{i=1}^n \inner{\calh}{v_L, v_i} \inner{\calh}{v_i, f},
\end{align*}
where the terms $\inner{\calh}{v_i, f}$ are not directly accessible. 

Computing this change of basis is trivial, since the matrix of change of basis 
can be shown to be lower triangular (see \cite{Pazouki2011}). We remark moreover that an efficient implementation of the whole greedy approximation process to 
compute 
$X_n$ and $\{w_i^*\}_{i=1}^n$ (or $\{\alpha_i\}_{i=1}^n$, as in Proposition \ref{prop:interpolation})  via $f/P$-greedy is 
available in Matlab \cite{VKOGA_matlab} and in 
Python \cite{VKOGA_python}. 
\end{rem}

Finally, for later use we remark also that \eqref{eq:energy_splitting_0} and  \eqref{eq:energy_splitting_1} prove that for all $X_n:=X_{n-1}\cup\{x_n\}$ there 
is a splitting of the interpolation error in the form
\begin{align}\label{eq:energy_splitting}
\norm{\calh}{v - \Pi_{X_n} v}^2 
&= \norm{\calh}{v - \Pi_{X_{n-1}} v}^2 - \left(\frac{v(x) -(\Pi_{X_{n-1}} v)(x)}{P_{n-1}(x)}\right)^2,
\end{align}
which represents a kind of energy splitting between the interpolant and the residual.

\section{Convergence for greedy points and continuous kernels}\label{sec:convergence-greedy}
For the greedy algorithm of Definition \ref{def:greedy_alg}, or equivalently the $f/P$-greedy algorithm applied to 
$v_L$, we can now derive convergence results. The proof follows the lines of the one of Theorem 3.7 in \cite{DeVore1996}, and in particular it makes use of 
the following proposition.
\begin{prop}[Lemma 3.4 in \cite{DeVore1996}]\label{lemma:temlyakov}
If $\{a_n\}_{n\in\N}$ is a sequence of non-negative numbers such that for a given $A>0$ it holds $a_1\leq A$ and $a_{n+1}\leq a_n (1 - a_n/A)$, then $a_n\leq A 
n^{-1}$ 
for all $n\in\N$. 
\end{prop}

The theorem proves that the worst case error converges to zero at least as $n^{-1/2}$. 
As we will explain when running numerical experiments in Section \ref{sec:numerics}, this has to be understood as a preliminary result, since faster 
convergence is often observed in practice. Nevertheless, this speed of convergence is independent of the input space dimension and, as we will comment 
later, it is better than the one that we proved for uniform points in certain cases. 

We remark that convergence of kernel-based integral approximations based instead on $P$-greedy points has been recently analyzed in \cite{Kanagawa2019b}.

\begin{theorem}[Convergence rates for greedy quadrature]\label{th:greedy_quad}
For any continuous s.p.d. kernel on a bounded and Lebesgue measurable set $\Omega\subset\R^d$, let $L\in\calh'$ be a linear functional such that there exists 
$1\leq q\leq\infty$ and 
$c_L\geq 
0$ such that
\begin{align*}
|L(f)|\leq c_L \norm{L_q(\Omega)}{f}\;\;\fa\;\; f\in\calh, 
\end{align*}
and let $v_L\in\calh$ be its Riesz representer.

Then, for the sequence of points $\{X_n\}_{n\in\N}$ selected by the greedy algorithm it holds
\begin{align*}
e_{\calh}\left(Q_{X_n, L}\right) = \norm{\calh}{v_L - \Pi_{X_n} v_L}\leq c_G n^{-1/2} \;\;\fa\;\; n\in\N,
\end{align*}
where $c_G:=\max\left\{\norm{\calh}{v_L}, c_{L} \left|\Omega\right|^{1/q}\max\limits_{x\in\Omega}{\sqrt{K(x, x)}} \right\}$.
\end{theorem}
\begin{proof}
For notational simplicity we denote the residual by $r_n:= v_L - \Pi_{X_n} v_L$. We assume that $r_n\neq 0$ for any finite $n$, otherwise we are done. 
Then 
equation \eqref{eq:energy_splitting} for $v_L$ reads
\begin{align}\label{eq:greedy_intermediate3}
\norm{\calh}{r_n}^2 
&= \norm{\calh}{r_{n-1}}^2 - \frac{\left|r_{n-1}(x_n)\right|^2}{P_{n-1}(x_n)^2}
= \norm{\calh}{r_{n-1}}^2 \left(1 - \frac{\left|r_{n-1}(x_n)\right|^2}{P_{n-1}(x_n)^2 \norm{\calh}{r_{n-1}}^2}\right),
\end{align}
and by the definition \eqref{eq:f_P_greedy} of the $f/P$-greedy selection we have that 
\begin{align}\label{eq:greedy_intermediate1}
\frac{\left|r_{n-1}(x_n)\right|}{P_{n-1}(x_n)} &=\max\limits_{x\in\Omega\setminus X_{n-1}} \frac{\left|r_{n-1}(x)\right|}{P_{n-1}(x)}.
\end{align}
Now let $\bar x\in\Omega$ be such that
\begin{align}\label{eq:greedy_intermediate2}
\bar x:= \argmax\limits_{x\in\Omega} \left|r_{n-1}(x)\right| = \argmax\limits_{x\in\Omega\setminus X_{n-1}} \left|r_{n-1}(x)\right|,
\end{align}
where the two maxima are equal since $r_{n-1} = 0$ on $X_{n-1}$ and thus $\bar x\notin X_{n-1}$. Since $\bar x\notin X_{n-1}$ it also holds $P_{n-1}(\bar 
x)\neq 0$, and then using \eqref{eq:greedy_intermediate1} and \eqref{eq:greedy_intermediate2} we have 
\begin{align}\label{eq:greedy_intermediate4}
\nonumber\frac{\left|r_{n-1}(x_n)\right|}{P_{n-1}(x_n)}
&=\max\limits_{x\in\Omega\setminus X_{n-1}} \frac{\left|r_{n-1}(x)\right|}{P_{n-1}(x)}
\geq \frac{\left|r_{n-1}(\bar  x)\right|}{P_{n-1}(\bar x)}
= \frac{\norm{L_{\infty}(\Omega)}{r_{n-1}}}{P_{n-1}(\bar x)}\\
&\geq \frac{\norm{L_{\infty}(\Omega)}{r_{n-1}}}{\norm{L_{\infty}(\Omega)}{P_{n-1}}}.
\end{align}
Moreover, since $Id - \Pi_V = \Pi_{V^{\perp}}$ for any closed subspace $V\subset\calh$, and using again the fact that orthogonal projections are self adjoint, 
we have
\begin{align*}
\norm{\calh}{r_{n-1}}^2 
&= \inner{\calh}{r_{n-1}, r_{n-1}}
= \inner{\calh}{v_L - \Pi_{X_{n-1}} v_L, r_{n-1}}
= \inner{\calh}{\Pi_{V(X_{n-1})^{\perp}} v_L, r_{n-1}}\\
&= \inner{\calh}{v_L, \Pi_{V(X_{n-1})^{\perp}}  r_{n-1}}
= \left( L \circ \Pi_{V(X_{n-1})^{\perp}}\right) (r_{n-1}),
\end{align*}
and using Proposition \ref{prop:restriction} and the boundedness of $\Omega$ we can thus control the norm of the residual as 
\begin{align*}
\norm{\calh}{r_{n-1}}^2 
&\leq c_{L \circ \Pi_{V(X_{n-1})^{\perp}}} \norm{L_q(\Omega)}{r_{n-1}}
\leq c_{L} \norm{L_q(\Omega)}{r_{n-1}}
\leq c_{L} \left|\Omega\right|^{1/q} \norm{L_{\infty}(\Omega)}{r_{n-1}},
\end{align*}
i.e.,
$\norm{L_{\infty}(\Omega)}{r_{n-1}} \geq {\norm{\calh}{r_{n-1}}^2}/({c_{L} \left|\Omega\right|^{1/q}})$.
Combining this bound and inequality \eqref{eq:greedy_intermediate4} we can continue to obtain
\begin{align*}
\frac{\left|r_{n-1}(x_n)\right|^2}{P_{n-1}(x_n)^2 \norm{\calh}{r_{n-1}}^2}
&\geq\frac{\norm{L_{\infty}(\Omega)}{r_{n-1}}^2}{\norm{L_{\infty}(\Omega)}{P_{n-1}}^2 \norm{\calh}{r_{n-1}}^2}
\geq \frac{\norm{\calh}{r_{n-1}}^4}{\norm{L_{\infty}(\Omega)}{P_{n-1}}^2 \norm{\calh}{r_{n-1}}^2 c_{L}^2 \left|\Omega\right|^{2/q}}\\
&= \frac{\norm{\calh}{r_{n-1}}^2}{\norm{L_{\infty}(\Omega)}{P_{n-1}}^2  c_{L}^2 \left|\Omega\right|^{2/q}}.
\end{align*}
Now we can set $A_{n-1}:= {\norm{L_{\infty}(\Omega)}{P_{n-1}}^2  c_{L}^2 \left|\Omega\right|^{2/q}}$ and, since $P_{n}$ is non 
increasing in $n$, we have
\begin{align*}
A_{n-1} &\leq A_0 =   c_{L}^2 \left|\Omega\right|^{2/q}\norm{L_{\infty}(\Omega)}{P_{0}}^2 
= c_{L}^2 \left|\Omega\right|^{2/q}\max\limits_{x\in\Omega}{K(x, x)},
\end{align*}
and thus the last inequality reads
\begin{align*}
\frac{\left|r_{n-1}(x_n)\right|^2}{P_{n-1}(x_n)^2 \norm{\calh}{r_{n-1}}^2}
&\geq \frac{\norm{\calh}{r_{n-1}}^2}{A_{0}}.
\end{align*}
Inserting this result in \eqref{eq:greedy_intermediate3} and defining $A:= \max\left\{\norm{\calh}{v_L}^2, A_0 \right\}$, we finally have
\begin{align*}
\norm{\calh}{r_n}^2 
&\leq \norm{\calh}{r_{n-1}}^2 \left(1 - \frac{\norm{\calh}{r_{n-1}}^2}{A_{0}}\right)
\leq \norm{\calh}{r_{n-1}}^2 \left(1 - \frac{\norm{\calh}{r_{n-1}}^2}{A}\right)
\end{align*}
and 
\begin{align*}
\norm{\calh}{r_0}^2 = \norm{\calh}{v_L}^2 \leq A,
\end{align*}
and thus the result follows by applying Lemma \ref{lemma:temlyakov} with $a_n:= \norm{\calh}{r_{n-1}}^2$.
\end{proof}

Before discussing some consequence of this result, we show in the following corollary that the same proof idea applies also to other greedy interpolation 
strategies which are commonly used in kernel interpolation and quadrature. 

\begin{cor}[Convergence for other selection rules]
Under the same hypotheses of Theorem \ref{th:greedy_quad}, the same convergence result holds if the $f/P$-greedy selection rule \eqref{eq:f_p_greedy_selection} 
is replaced by one of the following ones:
\begin{enumerate}[a)]
 \item The $f$-greedy selection rule, i.e., 
 \begin{align*}
x_n\in\argmax\limits_{x\in \Omega\setminus X_{n-1}} \left|v(x) - (\Pi_{X_{n-1}} v)(x)\right|.
 \end{align*}
 \item The selection rule of Algorithm~4 on p.\@~93  in \cite{Oettershagen2017}, i.e.,
 \begin{align*}
x_n\in\argmax\limits_{x\in \Omega\setminus X_{n-1}} \frac{\left|v(x) -(\Pi_{X_{n-1}} v)(x)\right|}{\sqrt{K(x, x)}}.
 \end{align*}
\end{enumerate}

\end{cor}
\begin{proof}
The proof of Theorem \ref{th:greedy_quad} depends on the $f/P$-greedy selection rule only via equation \eqref{eq:greedy_intermediate4} (and on equation 
\eqref{eq:greedy_intermediate1}, which is nevertheless used only to derive the latter). We thus only need to show how to obtain the same bound as in 
\eqref{eq:greedy_intermediate4} with these other selection rules:
\begin{enumerate}[a)]
 \item In this case, by the definition of $x_n$, it immediately holds that 
 \begin{align*}
\frac{\left|r_{n-1}(x_n)\right|}{P_{n-1}(x_n)}
&= \frac{\norm{L_{\infty}(\Omega)}{r_{n-1}}}{P_{n-1}(x_n)}
\geq \frac{\norm{L_{\infty}(\Omega)}{r_{n-1}}}{\norm{L_{\infty}(\Omega)}{P_{n-1}}}.
\end{align*}
\item For this selection rule, using the same argument as in the proof of Theorem \ref{th:greedy_quad} one obtains
 \begin{align*}
\frac{\left|r_{n-1}(x_n)\right|}{P_{n-1}(x_n)}
&\geq \frac{\norm{L_{\infty}(\Omega)}{r_{n-1}}}{\norm{L_{\infty}(\Omega)}{\sqrt{K(x, x)}}}
\end{align*}
in place of \eqref{eq:greedy_intermediate4}, and then the following of the proof is simplified since $\sqrt{K(x, x)}=P_0(x)$ for all $x\in\Omega$.
\end{enumerate}
\end{proof}

\begin{rem}[Other selection criteria]
We remark that the proof does not directly apply instead to the power scaled residual greedy selection (psr-greedy) which has been recently introduced in 
\cite{Dutta2020}. In this case the new point is selected as 
\begin{align*}
x_n\in\argmax\limits_{x\in \Omega\setminus X_{n-1}} P_{n-1}(x) \left|v(x) - (\Pi_{X_{n-1}} v)(x)\right|.
\end{align*}
The present approach fails in this case. Indeed, following the same idea, one needs to bound from below the ratio $P_{n-1}(\bar x) / P_{n-1}(x_{n-1})$, where 
$\bar x$ is defined as in \eqref{eq:greedy_intermediate2}, and this quantity may be arbitrarily small.
\end{rem}

Theorem~\ref{th:greedy_quad} is a first result on the convergence of the algorithm, and it has the advantage of providing rates of convergence that do not 
depend on the input dimension. 
Nevertheless, it is far from optimal in the sense that the ideal result that one can aim for in the setting of adaptive algorithms is rather the following: 
Given a 
functional $L$ such that there exist a sequence of possibly unknown (or not computable) point sets $\{\bar X_n\}_{n\in\N}$, and $t, C>0$ such that 
\begin{align*}
\norm{\calh}{v_L - \Pi_{\bar  X_n} v_L } \leq C n^{-t}\;\;\fa n\in\N,
\end{align*}
find a constructive algorithm that selects points $\{X_n\}_{n\in\N}$ such that 
\begin{align*}
\norm{\calh}{v_L - \Pi_{X_n} v_L} \leq C' n^{-t}\;\;\fa n\in\N,
\end{align*}
with a possibly larger constant $C'\geq C$. There is no reason to expect that this can be achieved by greedy algorithms in particular, but this is actually the 
case for 
other types of greedy algorithms (see e.g. \cite{Binev2010}), and they are furthermore particularly attractive for their easiness of implementation.

In the direction of this optimal expectation, it should also be mentioned that our algorithm is closely related to the greedy algorithms of 
\cite{Temlyakov2008}, where 
a dictionary in a generic Hilbert space $H$ is used to approximate a single function $f\in H$. In this case, the paper \cite{DeVore1996} proves that for every 
Hilbert 
space there exist a dictionary and a function in the class of Example \ref{example:temlyakov_class} such that the rate of $n^{-1/2}$ can not be improved. 
Nevertheless, 
in the case of the current paper we are using not a generic dictionary, but rather the particular one that is generated by translates of the reproducing 
kernel, and thus there 
is no reason to believe that the result of Theorem \ref{th:greedy_quad} can not be improved.

Despite being not optimal, we remark that also the rates of convergence proved here are better than the ones of Theorem \ref{th:uniform_quad} in certain cases.
Observe that this of course does not mean, even in this case, that greedy points are better than uniform points, but only that the estimate is better. 
\begin{prop}[Comparison of the rates for greedy and uniform points]\label{prop:compare_exponents}
Under the assumptions of Theorem \ref{th:uniform_quad}, the rates of convergence of Theorem \ref{th:greedy_quad} are better than the ones of Theorem 
\ref{th:uniform_quad} if
\begin{align*}
\frac{d}2<\tau  < \frac{d}2 + \frac{d}2 - \frac{d}q = d\left(1 - \frac1q\right).
\end{align*}
If $L$ is an integration functional as in Example \ref{example:measure}, this is equivalent to requiring that $d/2<\tau<d/p$ if $\nu\in L_p(\Omega)$.
\end{prop}
\begin{proof}
Since $\tau>d/2$, the proved rate of convergence of the greedy algorithm is better if
\begin{align}\label{eq:rem_comparison1}
{-\frac12} < -\frac{\tau}{d} +\left(\frac12 - \frac1q\right)_+,
\end{align}
and this can happen only if $\ \left(\frac12 - \frac1q\right)_+ >0$, i.e., $q>2$. 
In this case we have $\left(1/2 - 1/q\right)_+ = 1/2 - 1/q\in(0, 1/2]$, and then \eqref{eq:rem_comparison1} guarantees that the greedy estimate is better 
for all $d$ and $\tau$ such that 
\begin{align*}
\frac{d}2<\tau  < \frac{d}2 + \frac{d}2 - \frac{d}q = d\left(1 - \frac1q\right).
\end{align*}
In the particular case of integration functionals, this is equivalent to require that $d/2<\tau<d/p$ if $\nu\in L_p(\Omega)$.
\end{proof}

We conclude this section with some remarks on the results.

\begin{rem}[Almost optimal rate when $\tau \approx d/2$]
If $\Omega \subset \R^d$ is sufficiently regular, it is known that the rate $n^{-\tau/d}$ is optimal for approximation of $L(f) := \int_{\Omega} f(x) dx$ for 
functions in the Sobolev space $W_2^\tau(\Omega)$ with $\tau > d/2$ the rate. That is, there is no sequence of quadrature rules whose worst case error decays 
faster than this; see~\cite[Section~1.3.11]{Novak1988} and~\cite[Section~4.2.4]{NovakWozniakowski2008}.
Theorem~\ref{th:greedy_quad} therefore shows that the greedy quadrature algorithm is almost optimal if $\tau$ is close to $d/2$.
\end{rem}

\begin{rem}[Monotonicity and a-posteriori error estimation]
Observe that, since the greedy points are nested, the worst case error $\norm{\calh}{v_L - \Pi_{X_n} v_L}$ is strictly decreasing. This nevertheless does 
not 
mean that 
for one fixed function $f\in\calh$ the error 
\begin{align*}
\seminorm{}{L(f) - Q_{X_n,L}(f)} 
\end{align*}
is decreasing. In particular, in the case of a single function $f$, it would be interesting and useful to derive a-posteriori error 
estimator that allows to stop the greedy selection when the desired accuracy is reached on $f$. 

\end{rem}

\begin{rem}[Interpolation error in the $L_{\infty}(\Omega)$ norm]
We have analyzed so far the relation between the $\calh$-norm interpolation error of $v_L$ and the worst case quadrature error of $L$. If another norm is 
considered, the 
relation between the two errors is no more in place, but still from the point of view of interpolation it makes sense to measure other types of errors. 

If for example we measure the interpolation error $\norm{L_r(\Omega)}{v_L - \Pi_{X_n}v_L}$ for some $1\leq r\leq q$, then for any $X_n\subset\Omega$ with 
$h_{X_n}\leq h_0$, using Theorem \ref{th:sampling_inequality} and the same procedure as in Theorem \ref{th:uniform_quad} we have
\begin{align*}
\norm{L_r(\Omega)}{v_L - \Pi_{X_n}v_L} \leq c_S c_E h_{X_n}^{\tau -d (1/2 - 1/r)_+} \norm{\calh}{v_L - \Pi_{X_n}v_L}.
\end{align*}
and the norm in the right hand side is the one that has been bounded in Theorem \ref{th:uniform_quad} and Theorem \ref{th:greedy_quad}. This means that, 
provided 
$h_{X_n}\leq h_0$, the error in an $L_r$-norm carries an additional factor depending on $h_{X_n}$. By definition, $h_n$ is decreasing as $n^{-1/d}$ in the 
case  of uniform point sequences, while it does not even need to decrease for greedily selected points, which can leave arbitrarily large holes in $\Omega$. 

Nevertheless, also in this case we can check when the proven greedy convergence rate is better than the proven convergence rates for uniform points, and in 
this case the 
range will be smaller. In particular, similar computations as in Proposition \ref{prop:compare_exponents} give that the greedy rates are better if 
\begin{align*}
\frac{d}{2}<\tau<\frac{d}{2} \left(\frac32 - \frac1r  - \frac1{q}\right).
\end{align*}
For example, for $q=r=\infty$ this holds if $\frac{d}{2}<\tau<\frac{3}{4}d$.

\end{rem}

\begin{rem}[Integration on manifolds]\label{rem:manifold}
We remark that both Theorem \ref{th:uniform_quad} and Theorem \ref{th:greedy_quad} remain valid also for certain sets $\Omega$ which are not flat subsets of 
$\R^d$.

Namely, Theorem \ref{th:greedy_quad} only assumes that $L_p$ spaces can be defined over $\Omega$, and this is a fairly general assumption. Otherwise, the space 
$\calh$ is treated as a generic Hilbert space, without particular links to the structure of the underlying subset of $\R^d$.

Theorem \ref{th:uniform_quad}, on the other hand, makes use of the error estimate of Theorem \ref{th:sampling_inequality}, which holds for $\Omega\subset\R^d$. 
Nevertheless, similar results exist for more general sets such as for manifolds embedded in $\R^s$ (see e.g. \cite{Wright2012}), and in this case the error 
rates depend on the dimension $d$ of the manifold. These kind of results can be used in the proof of the theorem, and we analyze 
an example of this setting in Section \ref{sec:exp_sphere}.

\end{rem}

\section{Perturbations of the Riesz representer}\label{sec:stability}

The construction of both a generic weight-optimal quadrature formula, and the greedy selection of the centers, are based on the interpolation of the Riesz 
representer 
$v_L$. The whole algorithm is thus based on the availability of evaluations of $v_L$, and this can be an unrealistic assumption in some cases. In particular, 
since 
$v_L = L^y(K(\cdot, y))$, we need to have an efficient and exact way to compute the functional on the kernel, which is often not possible for example in the 
case of 
integration on arbitrary sets $\Omega$. 

Instead, we can assume to have an accurate but expensive approximation $\tilde L$ of $L$, which can be evaluated but is too expensive to be practical 
when the evaluation of the integrand $f$ is expensive. Since the evaluation of the kernel is instead cheap, we can compute the exact Riesz representer of 
$\tilde L$, 
i.e., $\tilde v_L : = \tilde L^y K(\cdot, y)$, and use it as an approximation of the exact Riesz representer $v_L$.

For example, $\tilde L$ can be in the form of a high-accuracy quadrature rule $\tilde L(f) := \sum_{i=1}^M \rho_i f(z_i)$ where $M$ is very large. In this case 
computing $\tilde L(f)$ for a function $f$ which is expensive to evaluate is not a viable option, while we can construct $\tilde v_L:=\sum_{i=1}^M \rho_i 
K(\cdot, \tilde 
z_i)$.

As a first step, the following proposition states that the worst case error of the weight-optimal quadrature is stable with respect to the use of an 
approximated Riesz 
representer, provided that also the perturbed functional is continuous.

Similar results appear in~\cite[Appendix~B]{Briol2019} and~\cite[Section~2]{Sommariva2006} for numerical integration, although they are less sharp. For 
example, the former has $\sqrt{n} \varepsilon_L^2$ in the place of $\varepsilon_L^2$ in the upper bound. 

\begin{prop}[Stability]
Let $L, \tilde L\in\calh'$ be two functionals with Riesz representers $v_L, \tilde v_L\in\calh$, where $\tilde L$ is a perturbation of $L$ with  
\begin{align*}
\varepsilon_L:= \sup\limits_{\norm{\calh}{f}\leq 1} \seminorm{}{L(f) - \tilde L(f)}. 
\end{align*}
Let $X\subset \Omega$ and let $Q_{X, L}$ and $\tilde Q_{X, L}:=Q_{X, \tilde  L}$ be the weight-optimal quadrature rules obtained by the interpolation on $X$ of 
$v_L$ 
and $\tilde v_L$, respectively. 

Then the error obtained by using the wrong Riesz representer $\tilde v_L$ to approximate $L$ can be bounded as
\begin{align*}
\sup\limits_{\norm{\calh}{f}\leq 1}\seminorm{}{L(f) - \tilde Q_{X,L}(f)}^2 &= e_{\calh}(Q_{X, L})^2 + \varepsilon_L^2.
\end{align*}
\end{prop}
\begin{proof}
By linearity we have that $\delta_L:= v_L - \tilde v_L\in\calh$ is the Riesz representer of the functional $L-\tilde L\in\calh'$, and by assumption 
$\norm{\calh}{\delta_L}=\varepsilon_L$. We can then rewrite the worst-case error of the statement as required, i.e.,
\begin{align*}
\sup\limits_{\norm{\calh}{f}\leq 1}\seminorm{}{L(f) - \tilde Q_{X,L}(f)}^2 
&=\norm{\calh}{v_L - \Pi_X\left(\tilde v_L\right)}^2
=\norm{\calh}{v_L - \Pi_X  v_L + \Pi_X  \delta_L}^2\\
&=\norm{\calh}{v_L - \Pi_X  v_L}^2 + \norm{\calh}{\Pi_X  \delta_L}^2 + 2 \inner{\calh}{v_L - \Pi_X  v_L, \Pi_X  \delta_L}\\
&= e_{\calh}(Q_{X, L})^2 + \varepsilon_L^2+ 2 \inner{\calh}{v_L - \Pi_X  v_L, \Pi_X  \delta_L}\\
&= e_{\calh}(Q_{X, L})^2 + \varepsilon_L^2,
\end{align*}
where in the last step we used the fact that $\Pi_X  \delta_L\in V(X)$ and $v_L - \Pi_X  v_L\in V(X)^{\perp}$.
\end{proof}

We remark that if $\tilde{L}$ is a quadrature rule, then $\varepsilon_L$ is merely the worst-case error of this quadrature rule, and can thus be estimated. 
Moreover, for e.g. $\rho_i$ and $\tilde{z}_i$ the weights and points of an inexpensive (quasi) Monte Carlo rule, $\varepsilon_L$ decays with a well-known rate 
when $M$ 
increases.

This results is useful to quantify the error introduced by an approximated knowledge of $v_L$ only if $X$ is fixed. If instead the greedy algorithm is used, 
then also 
the points $X$ themselves depend crucially on $v_L$, and thus perturbations of the Riesz representer lead to different sequences of quadrature 
centers. To deal with this case we have the following result.
In this case we need to assume that $\tilde L$, and not $L$, is continuous on $L_q(\Omega)\cap \calh$, since it is the one used to run the greedy algorithm. 

\begin{prop}[Stability of the greedy quadrature]\label{th:stability_two}
Let $L, \tilde L\in\calh'$ be two functionals with Riesz representers $v_L, \tilde v_L\in\calh$, where $\tilde L$ is a perturbation of $L$ with  
\begin{align*}
\varepsilon_L:= \sup\limits_{\norm{\calh}{f}\leq 1} \seminorm{}{L(f) - \tilde L(f)}. 
\end{align*}
Assume furthermore that there exists $1\leq q\leq \infty$ and $\tilde c_L>0$ such that $\tilde L$ is continuous on $L_q(\Omega)\cap \calh$ with norm bounded by 
$\tilde 
c_L$. 

Let $X_n:=X_n(\tilde L)\subset \Omega$ be the set of points selected after $n$ iterations of the greedy algorithm applied to $\tilde v_L$, and let 
$\tilde Q_{X_n, L}:=Q_{X_n, \tilde  L}$ be the corresponding weight-optimal quadrature rule. 
Then
\begin{align*}
\sup\limits_{\norm{\calh}{f}\leq 1}\seminorm{}{L(f) - \tilde Q_{X_n,L}(f)}^2 &\leq  c_{G} n^{-1/2} \left(c_{G} n^{-1/2} + 2 \varepsilon_L\right) 
+ \varepsilon_L^2,
\end{align*}
where $c_G:=\max\left\{\norm{\calh}{\tilde v_L}, \tilde c_{L} \left|\Omega\right|^{1/q}\max\limits_{x\in\Omega}{\sqrt{K(x, x)}} \right\}$.
\end{prop}
\begin{proof}
For simplicity of notation we set $X:=X_n$, since $n$ is fixed here. Defining $\delta_L$ as in the previous proposition we have
\begin{align*}
\sup\limits_{\norm{\calh}{f}\leq 1}\seminorm{}{L(f) - \tilde Q_{X,L}(f)}^2 
&=\norm{\calh}{v_L - \Pi_X\tilde v_L}^2
=\norm{\calh}{\tilde v_L +\delta_L - \Pi_X\tilde v_L}^2\\
&=\norm{\calh}{\tilde v_L - \Pi_X\tilde v_L}^2 + \norm{\calh}{\delta_L}^2 + 2 \inner{\calh}{\tilde v_L - \Pi_X\tilde v_L, \delta_L}\\
&\leq \norm{\calh}{\tilde v_L - \Pi_X\tilde v_L}^2 + \norm{\calh}{\delta_L}^2 + 2 \norm{\calh}{\tilde v_L - \Pi_X\tilde 
v_L}\norm{\calh}{\delta_L}\\
&= \norm{\calh}{\tilde v_L - \Pi_X\tilde v_L} \left(\norm{\calh}{\tilde v_L - \Pi_X\tilde v_L} + 2 \norm{\calh}{\delta_L}\right) + 
\norm{\calh}{\delta_L}^2,
\end{align*}
and inserting the estimate of Theorem \ref{th:greedy_quad} for the functional $\tilde L$ gives the result.
\end{proof}

\section{Numerical experiments}\label{sec:numerics}

In this section we test the greedy algorithm on various integration test problems. We start by an example where greedy points provide the same rate of 
convergence of the worst case error as uniform points, but possibly with an arbitrarily better constant, and then we analyze a case where also the rate is 
strictly 
better for greedy points. Then, we show how the algorithm performs on a manifold, and finally we test the method on an Uncertainty Quantification benchmark 
problem. 

All the experiments use the Python implementation \cite{VKOGA_python} of the $f/P$-greedy algorithm.

\subsection{Integration with a compactly supported density}\label{sec:exp_square}
We consider the unit square $\Omega:=[0,1]^2\subset\R^2$ as input domain and a quadratic Matern kernel, which is defined as $K(x, y) := 
\phi(\|x-y\|)$ with 
\begin{align*}
\phi(r): = e^{-\gamma r} \left(3 + 3 \gamma r + (\gamma r) ^2\right), 
\end{align*}
with a free parameter which is set to the value $\gamma = 1$ in 
this experiment. The kernel corresponds to a value of $\tau = 4$ for the decay of the Fourier transform in \eqref{eq:fourier_exponent} (see e.g. Appendix D in 
\cite{Fasshauer2007}).

We consider  the integration functional
\begin{align*}
L(f):= \int_{\Omega} f(x) \nu(x) dx \;\;\fa\;\; f\in\calh,
\end{align*}
where $\nu$ is the indicator function of the square $[0.3,0.5]\times [0.6,0.8]\subset\Omega$. One would expect that optimal quadrature points for this 
functional are uniformly distributed inside the support of $\nu$.

We approximate $L$ with a Monte Carlo approximation $\tilde L$ that uses $M:=10^4$ uniformly randomly distributed 
points on the square, i.e.,
\begin{align*}
\tilde L(f) := \frac1M \sum_{i=1}^M f(x_i) \nu(x_i).
\end{align*}
This approximation is used to compute the Riesz representer $\tilde v_L$ which is used to construct the approximant, as discussed in Theorem 
\ref{th:stability_two}. 

The greedy algorithm is run by selecting points from a uniform grid $X_{tr}$ of $10^4$ points, and it is terminated when $500$ points are selected, or when 
the maximal absolute interpolation error on $X_{tr}$ is below the tolerance $10^{-12}$. In this way, $n=358$ points are selected, and they are shown in Figure 
\ref{fig:square_weights}. The points are colored according to the magnitude of the corresponding weight, and they are overlapped to the contour lines of the 
density $\nu$. 

\begin{figure}[h!]
\begin{center}
\includegraphics[width=\textwidth]{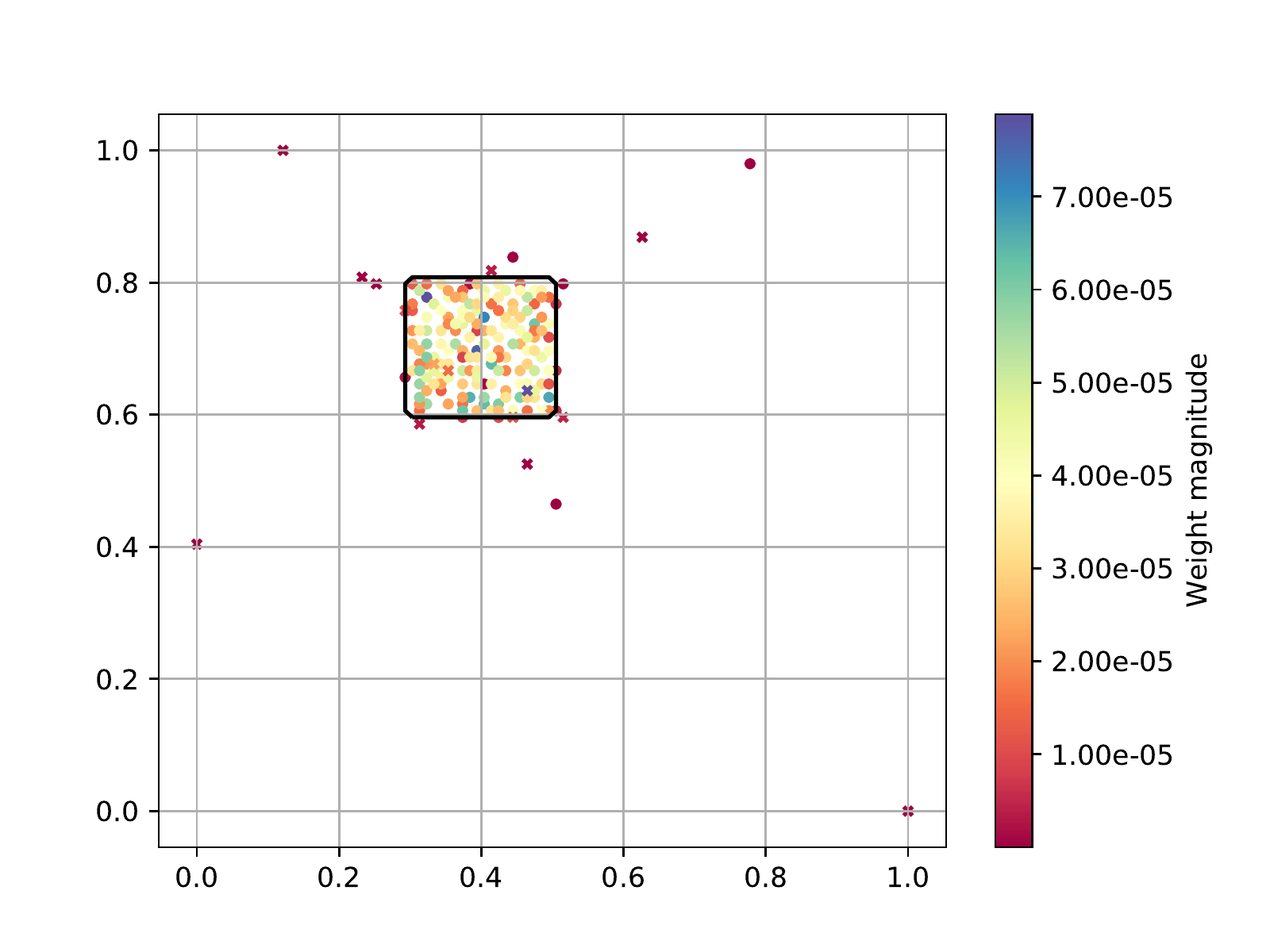}
\end{center}
\caption{
Results of the greedy algorithm in the example of Section \ref{sec:exp_square}. The figure shows the contour plot of the density $\nu$ (bold grayscale 
line), and the position of the points selected by the algorithm, which are colored according to the magnitude of the corresponding positive (circles) or 
negative 
(crosses) weight.
}
\label{fig:square_weights}
\end{figure}

The greedy algorithm selects almost all points inside the support of the density, which is the behavior one would 
expect from an optimal algorithm. Nevertheless, some points are selected in the area where $\nu=0$, even if the corresponding weights are relatively small.
We remark that this behavior may be caused by the use of the approximate Riesz representer.

Moreover, most of the weights are positive, and the negative ones are mostly of small magnitude.

In Figure \ref{fig:square_convergence} we show the worst-case error w.r.t. $\tilde L$ obtained with these greedy points. Since the points are nested, it is 
possible to show the decay of the worst case error as a function of the number of centers. As a comparison, we report also the decay of the worst case error 
for 
the optimal quadrature rule which uses uniform grids of points of increasing cardinality. The figure suggests that both approximation errors decay as 
$n^{-\tau/d}$, and in particular, after an initial phase, the convergence of the greedy error is faster than $n^{-1/2}$, and this suggests that indeed the rate 
of Theorem \ref{th:greedy_quad} is in general suboptimal. 

\begin{figure}[!h]
\begin{center}
\includegraphics[width=\textwidth]{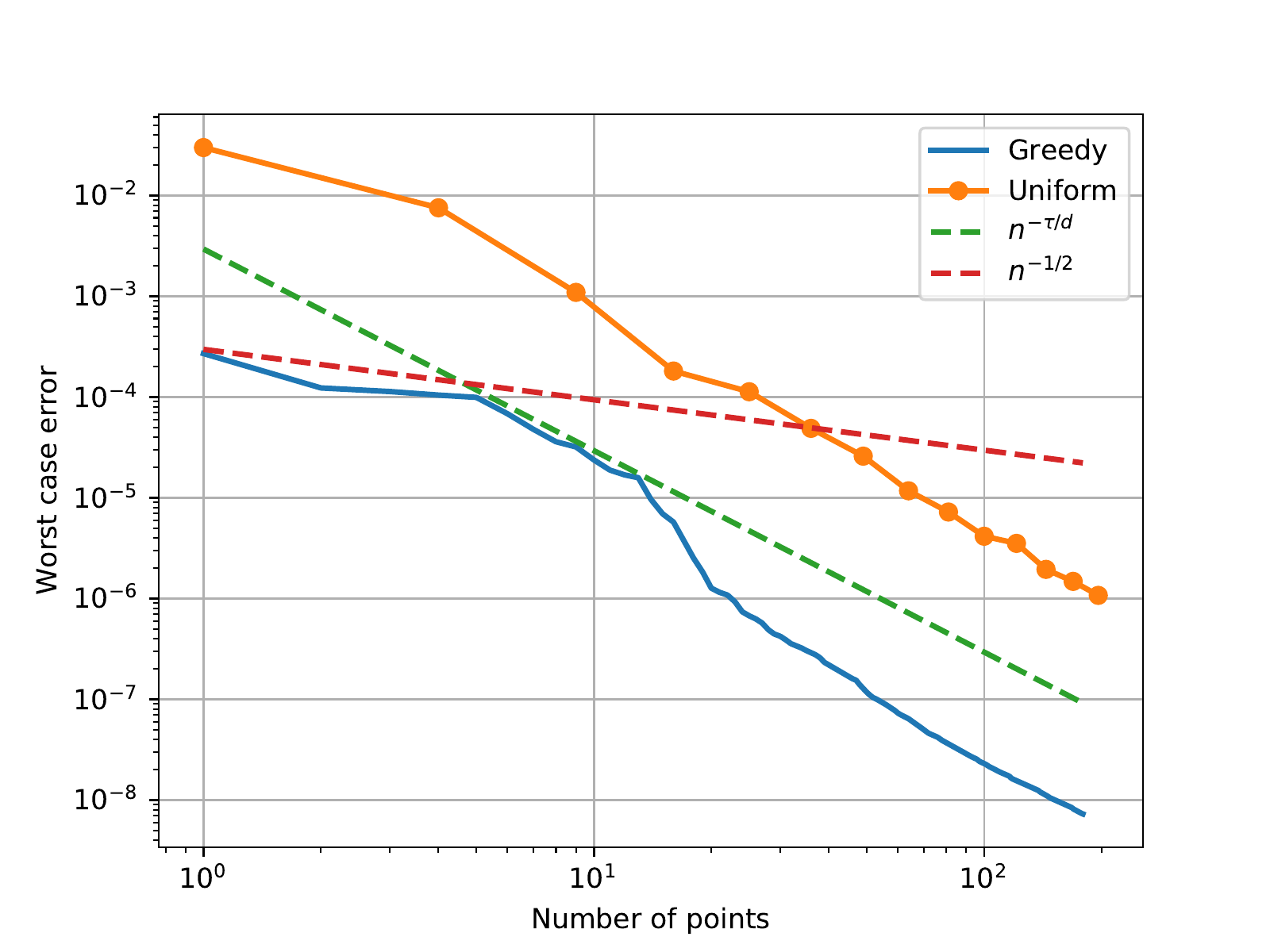}
\end{center}
\caption{
Decay of the worst-case error as a function of the number of points for the example of Section \ref{sec:exp_square}. The figure shows the error for the greedy 
points and for a grid of equally spaced points, and rates of decay scaled to the greedy error.
}
\label{fig:square_convergence} 
\end{figure}

Moreover, although the rate of convergence is the same for the two distribution of points in the asymptotic regime, the greedy error is smaller and the ratio 
$\rho$ between 
the two errors is roughly constant as a function of the number of points. This is due to the fact that the uniform points are 
filling the full $\Omega$, while 
the greedy ones fill only the support of $\nu$. We remark that this ratio $\rho$ can be made arbitrarily small by reducing the support of $\nu$, but no 
improvements 
should be expected in the rate of convergence by using greedy methods.

\subsection{Integration with a singular density}\label{sec:exp_singular}
We consider here $\Omega := [0,1]^2 \subset \R^2$ and $L(f) := \int_{\Omega} f(x) \nu(x) d x$, and set $\nu(x) = \|x-x_c\|^{-\alpha}$ with $\alpha 
\geq 0$ and $x_c:=[0.5, 0.5]^T$.

To obtain the best rate of convergence of the error in Theorem \ref{th:uniform_quad}, one should choose the smallest possible $q$ that satisfies the hypotheses 
of the theorem for a given $L$. Indeed, if $q$ is sufficiently small the term $(1/2 - 1/q)_+$ vanishes in the exponent. In this case, we can choose the 
smallest $q$ such that $\nu \in L_p(\Omega)$ with $1/p + 1/q = 1$, and since $\nu \in L_p(\Omega)$ if and only if $p < d/\alpha = 2/\alpha$, it follows 
that the limiting value is $q = 2 / (2 - \alpha)$, which gives $(1/2 - 1/q)_+ = (\alpha/2 - 1/2)_+ $. 

We thus consider values $\alpha \in\{1, 3/2, 2, 5/2\}$, and run the experiments with the same kernel, the same approximation of $L$, and the same setting 
of the greedy algorithm as in Section \ref{sec:exp_square}. The results are reported in Figure \ref{fig:singular}, where the points and weights of the greedy 
approximation are shown, together with the rate of convergence of the greedy and uniform quadrature as functions of the number of points. 

The figure clearly shows that the greedy points are increasingly concentrated around the singularity of $\nu$ as $\alpha$ increases, with weights which are 
larger. Moreover, the rates of convergence show that quadrature rules with uniform points have a worst case error converging to zero as $n^{-\tau/d + (\alpha/2 
- 1/2)_+}$, which is strictly slower than $n^{-\tau/d}$ for $\alpha>1$. The greedy algorithm, on the other hand, selects non uniform points and it provides a 
convergence like $n^{-\tau/d}$. Observe that in the case of $\alpha=2$, the convergence of the greedy algorithm seems to saturate. This is probably due to the 
fact that the greedy algorithm is implemented by selecting points from a fixed grid, and thus there is a limit on the maximal clustering that can be achieved 
around the singularity.

In other words, uniform points are sub-optimal for too skewed linear functionals.

\begin{figure}
\begin{center}
\begin{tabular}{ccc}
\rotatebox{90}{\hspace{1.8cm}$\boldsymbol{\alpha = 1}$}&
\includegraphics[width=.45\textwidth]{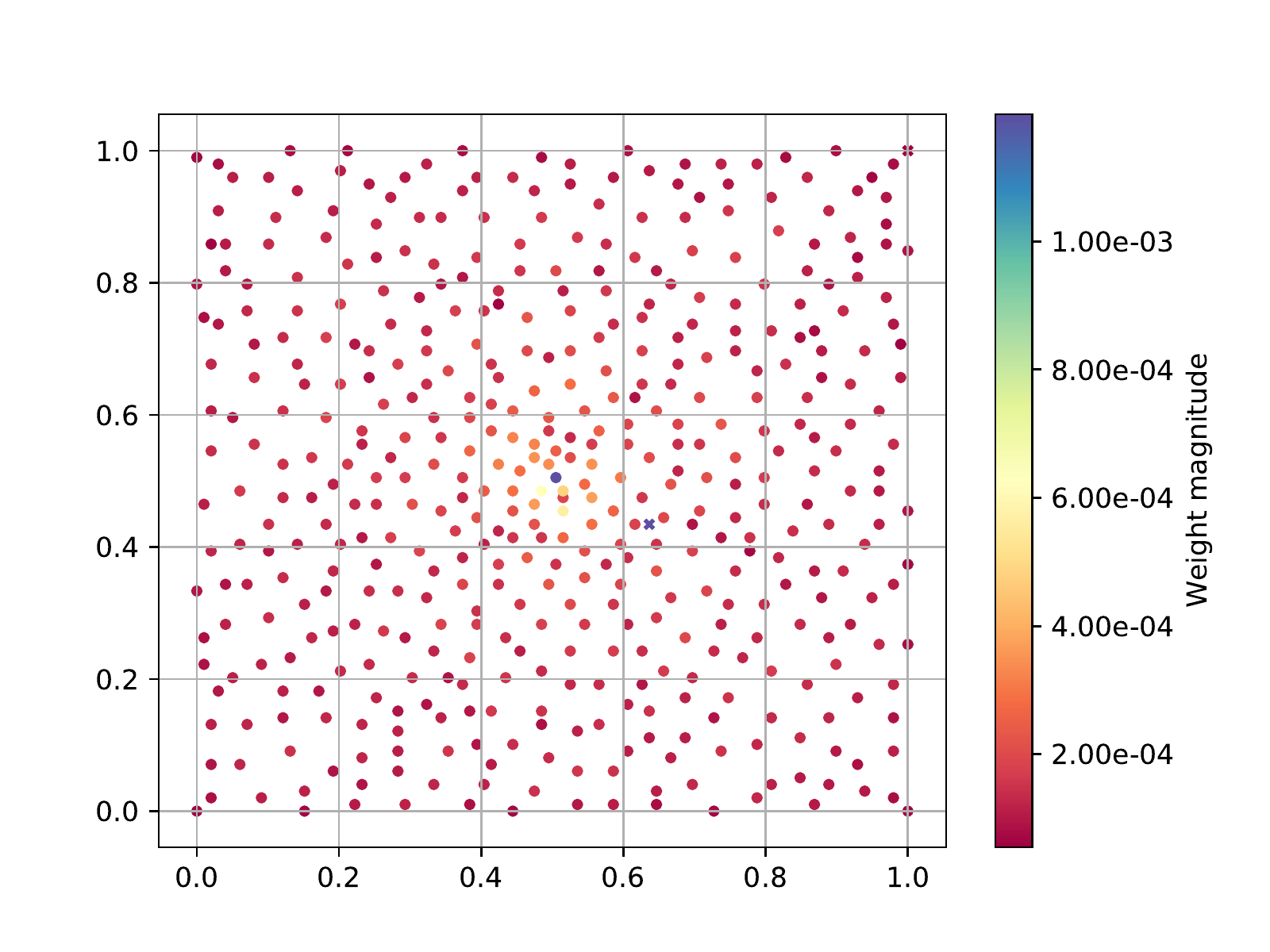}
&\includegraphics[width=.45\textwidth]{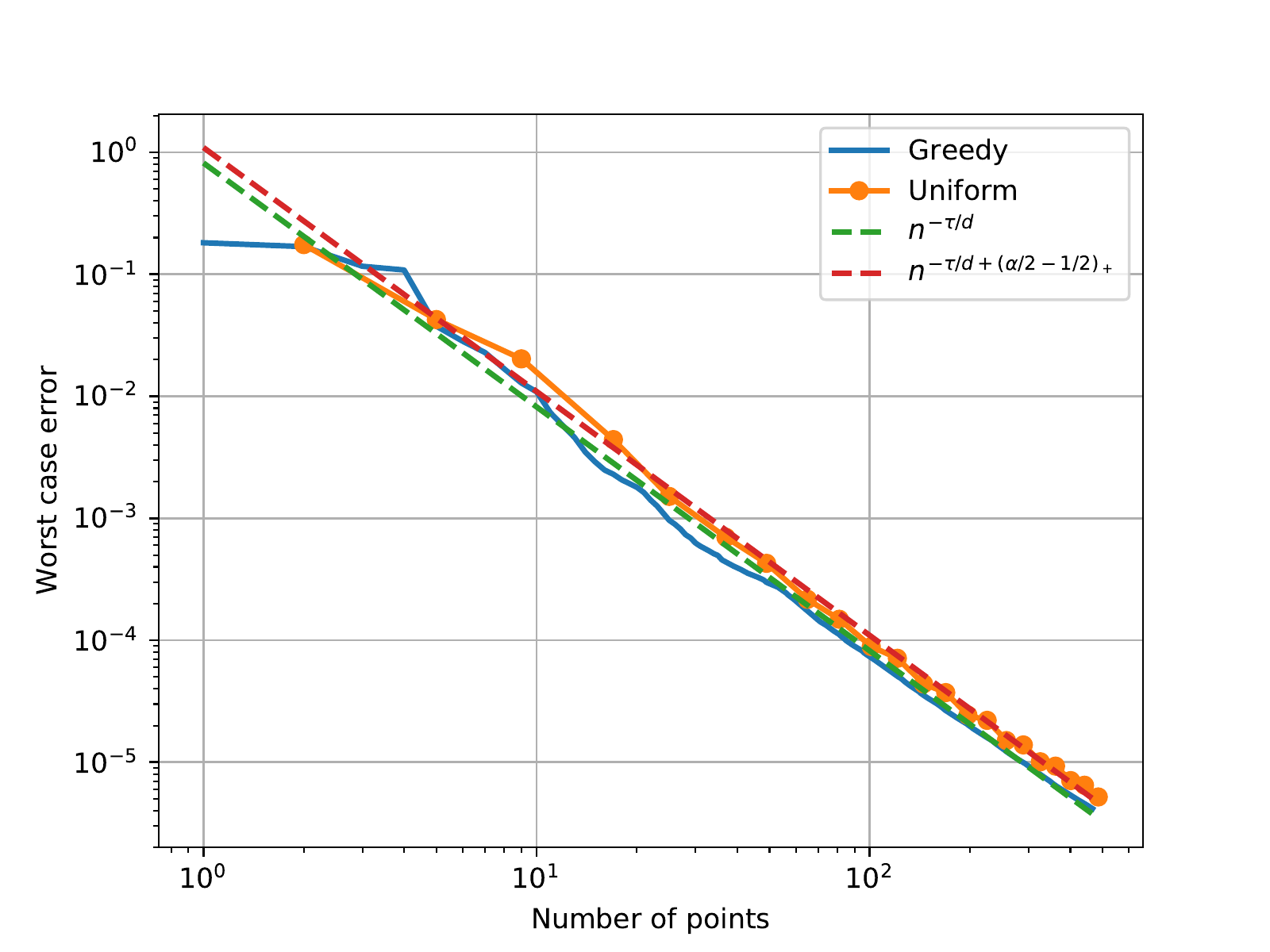}\\
\rotatebox{90}{\hspace{1.8cm}$\boldsymbol{\alpha = 3/2}$}&
\includegraphics[width=.45\textwidth]{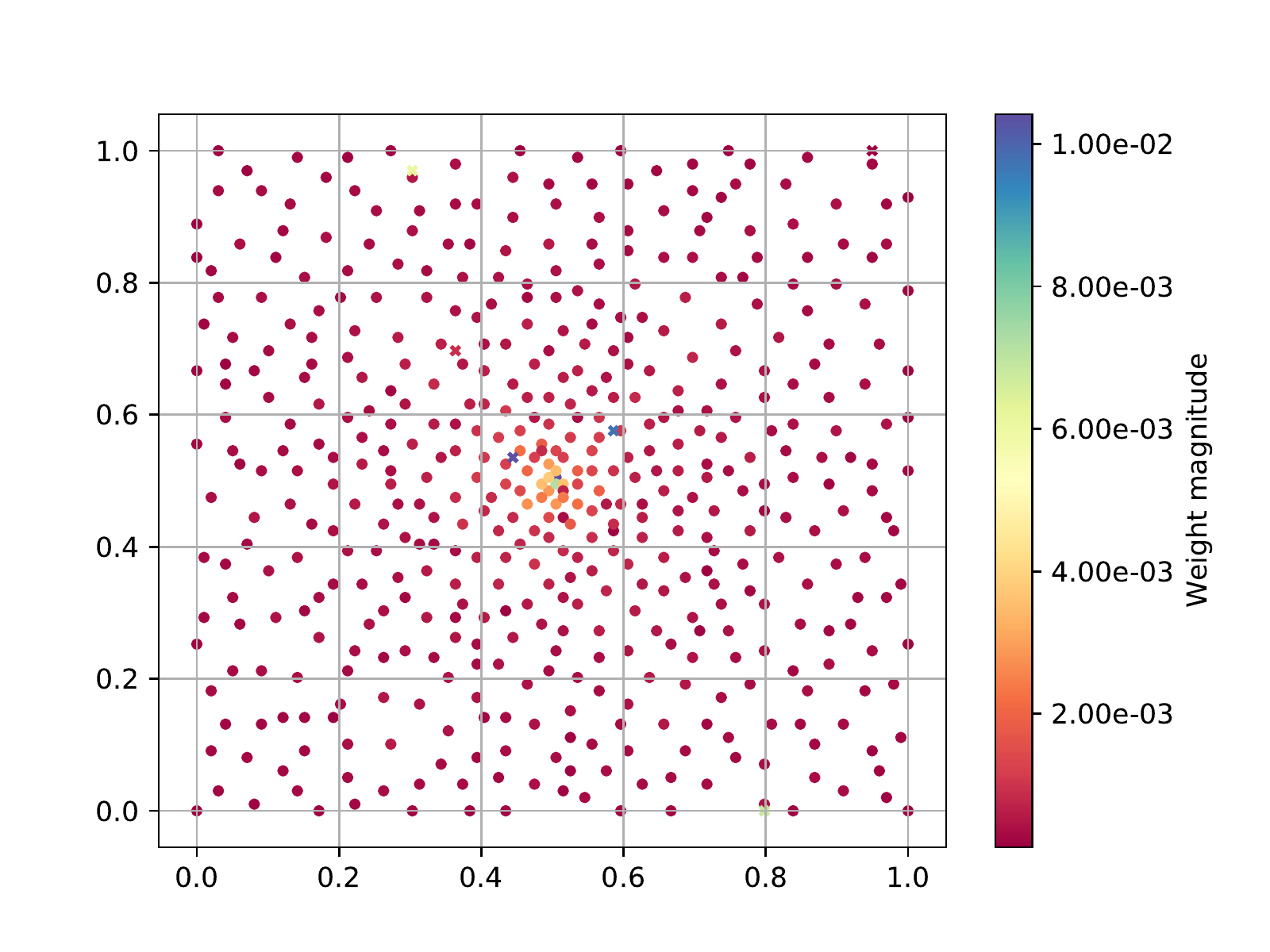}
&\includegraphics[width=.45\textwidth]{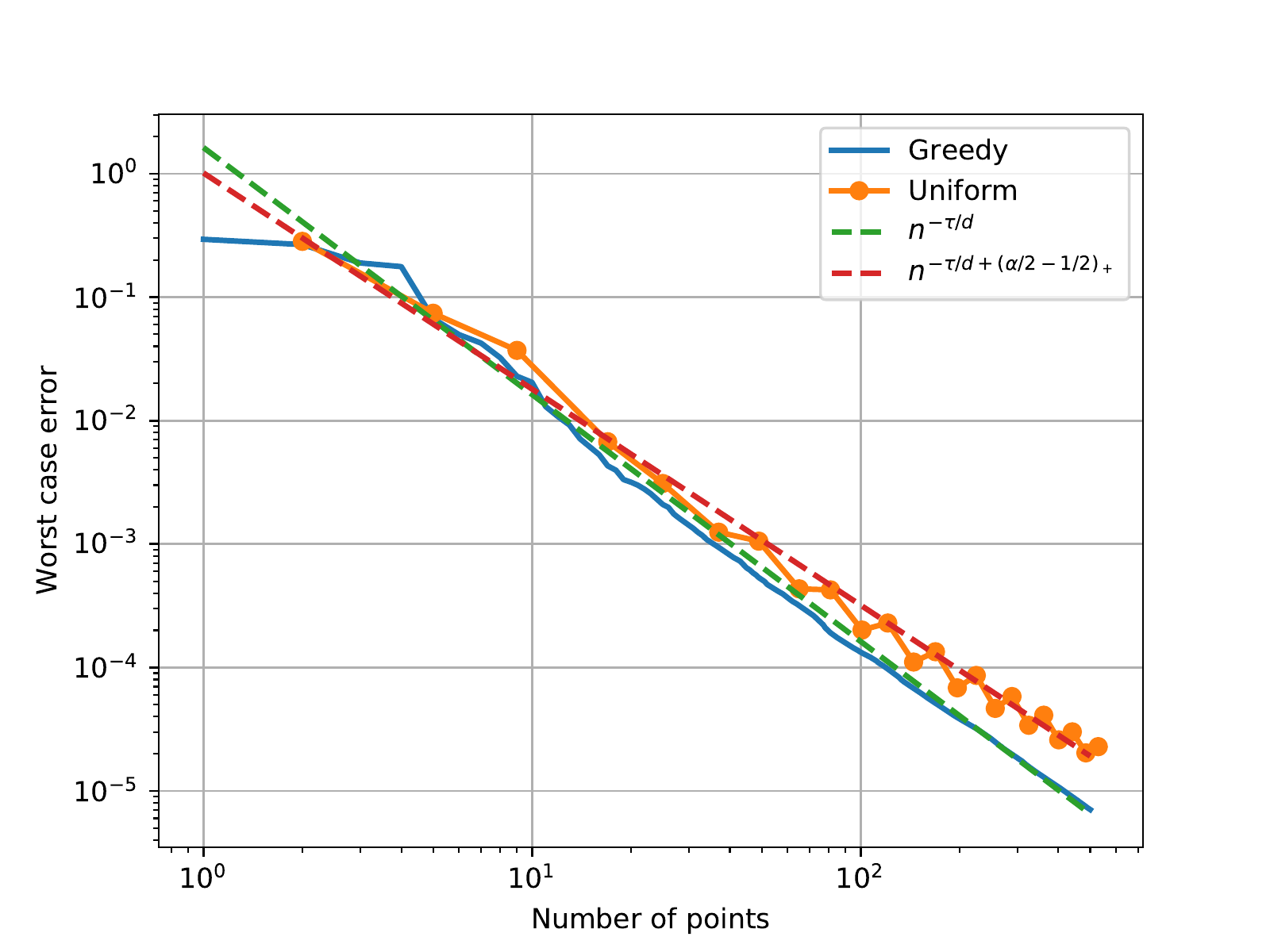}\\
\rotatebox{90}{\hspace{1.8cm}$\boldsymbol{\alpha = 2}$}&
\includegraphics[width=.45\textwidth]{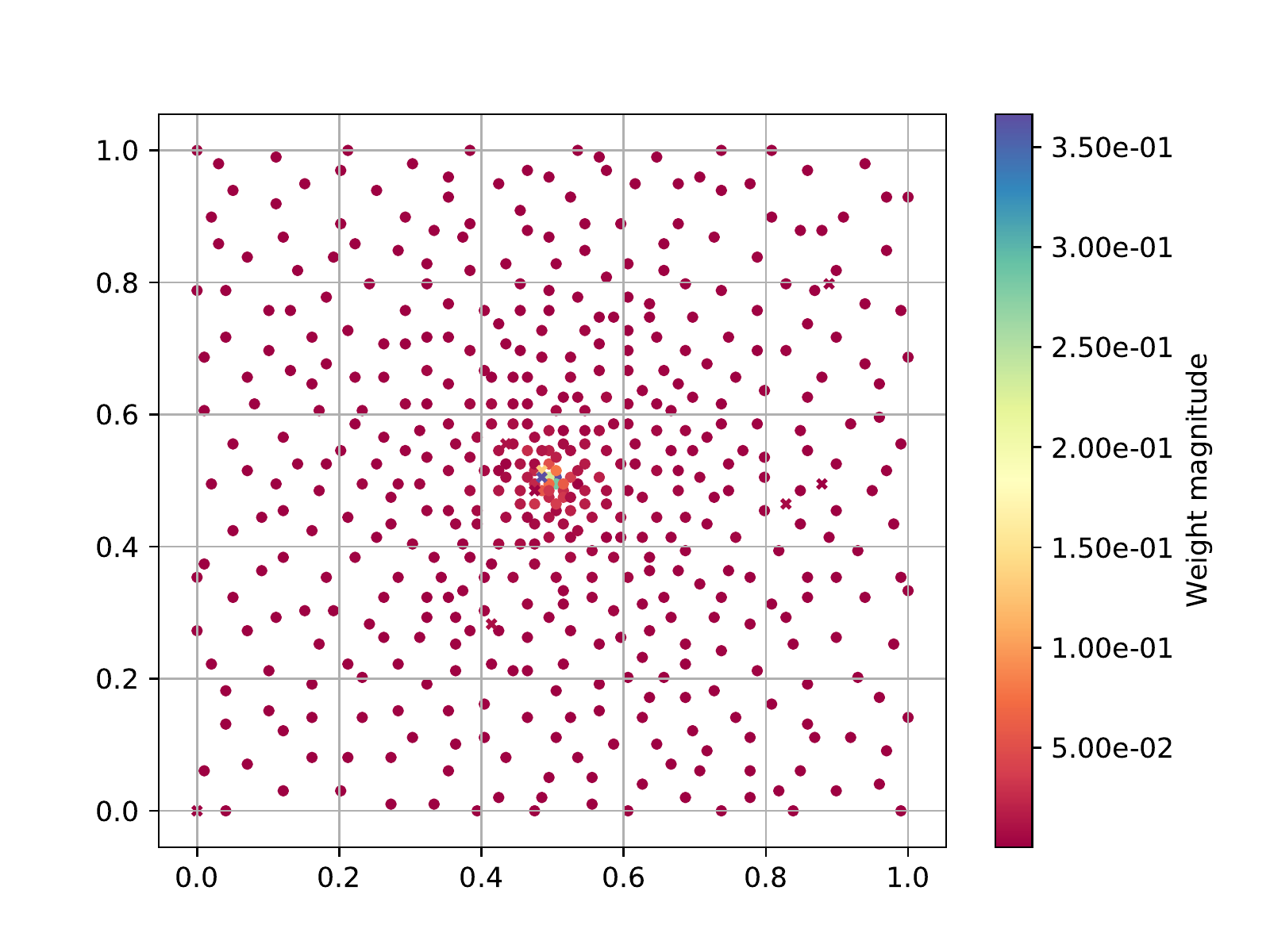}
&\includegraphics[width=.45\textwidth]{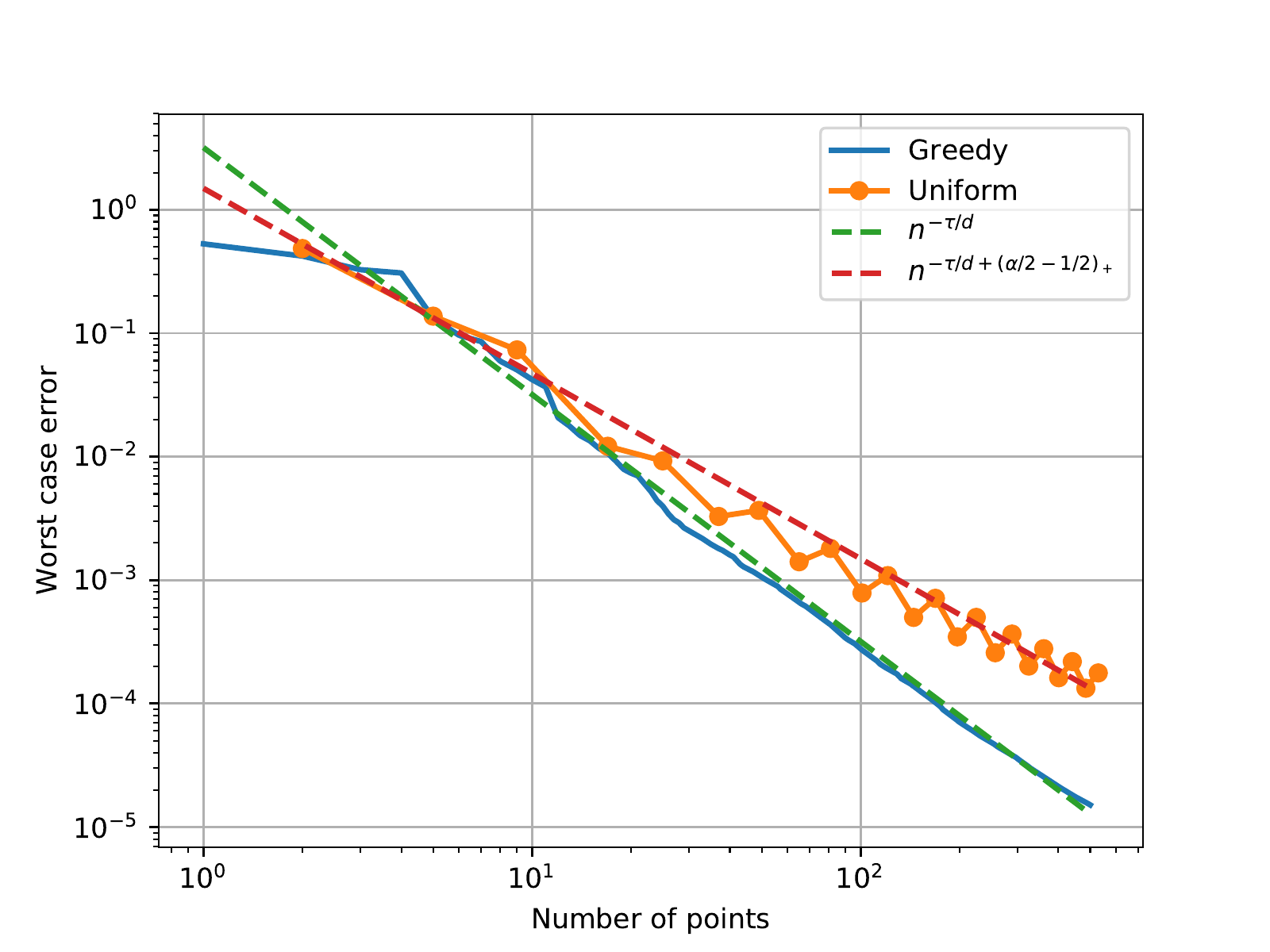}\\
\rotatebox{90}{\hspace{1.8cm}$\boldsymbol{\alpha = 5/2}$}&
\includegraphics[width=.45\textwidth]{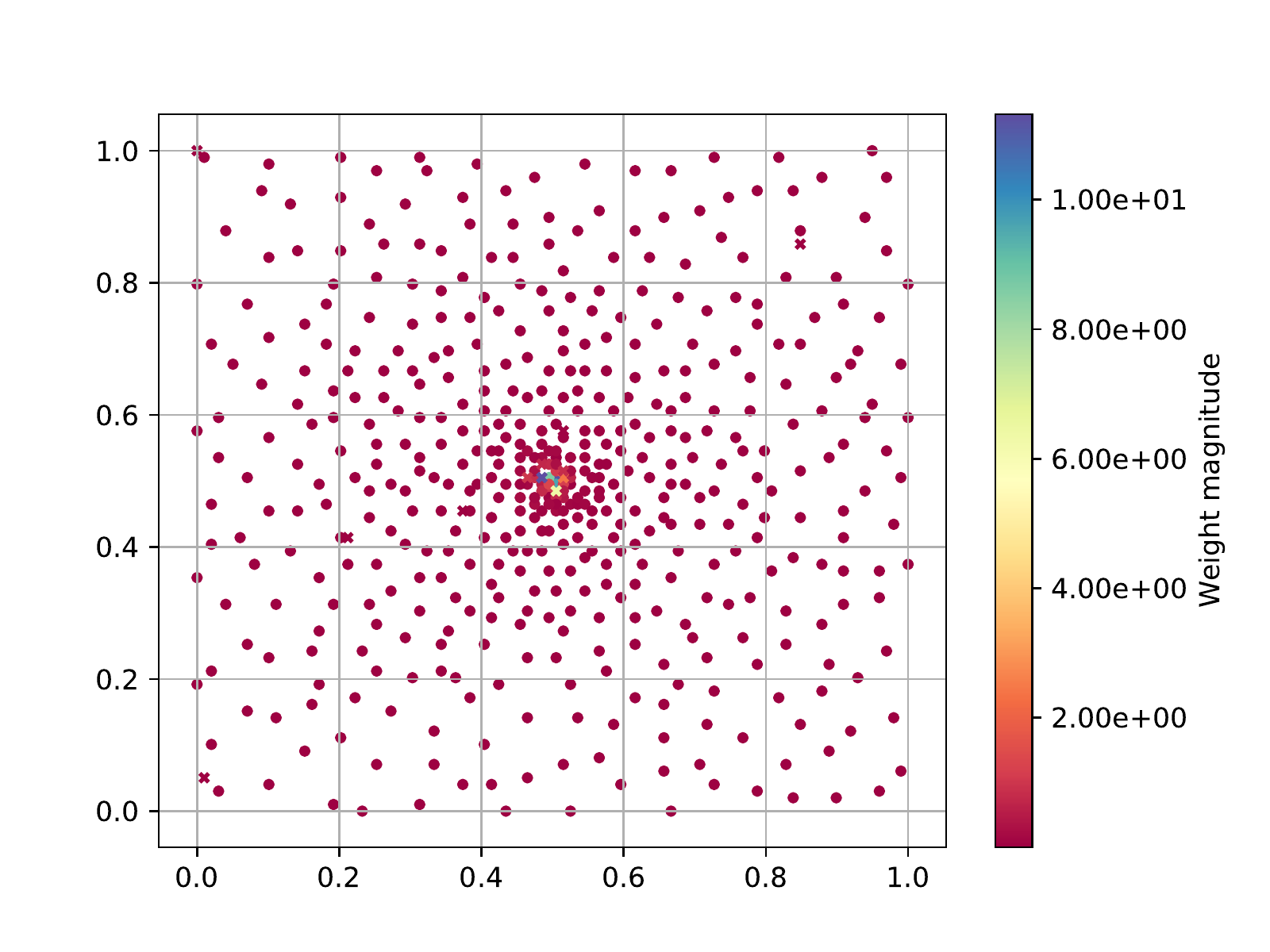}
&\includegraphics[width=.45\textwidth]{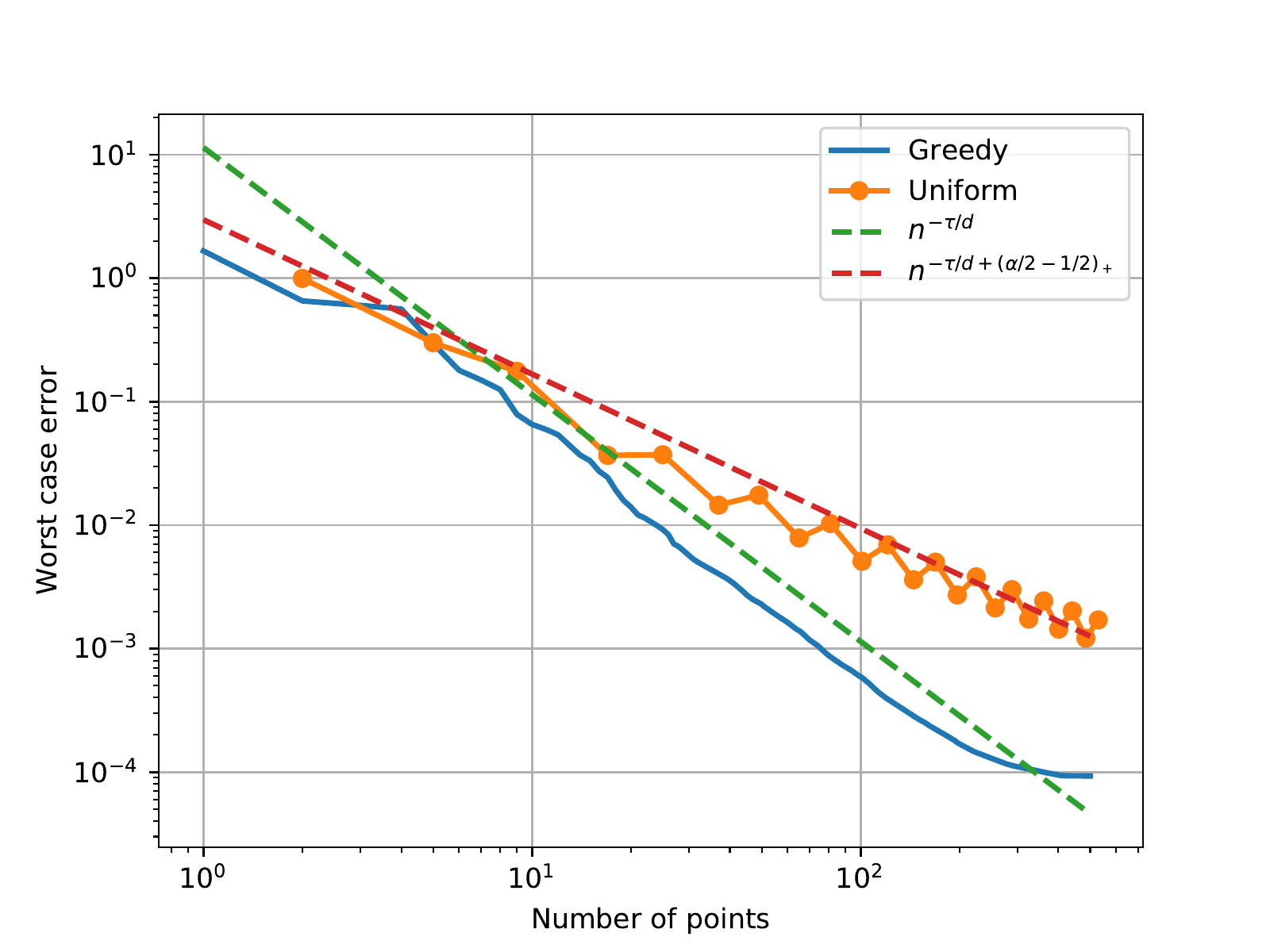}
\end{tabular}
\end{center}
\caption{
Results of the greedy algorithm in the example of Section \ref{sec:exp_singular} for values of $\alpha\in\{1, 3/2, 2, 5/2\}$ (from top to bottom). The left 
column shows the position of the points selected by the algorithm, which are colored according to the magnitude of the corresponding positive (circles) or 
negative 
(crosses) weight (left column). The right column shows the decay of the worst-case error as a function of the number of points, 
both for greedy and equally spaced points. The rates of decay are scaled by a constant.
}
\label{fig:singular}
\end{figure}

\subsection{Integration on a manifold}\label{sec:exp_sphere}

A similar experiment as in Section \ref{sec:exp_square} is repeated on the sphere $\Omega:=\mathbb S^2\subset\R^3$. Assuming $\Omega$ is represented in 
Cartesian coordinates, we 
consider the integration functional
\begin{align*}
L(f):= \int_{\Omega} f(x) \nu(x) dx \;\;\fa\;\; f\in\calh,
\end{align*}
where 
\begin{align*} 
\nu(x):= \exp\left((x - x_c)^T \Sigma (x - x_c) \right), \quad x_c:= [0, -1, 0]^T, 
\end{align*}
and $\Sigma$ is the diagonal matrix with diagonal $[-5, -5, -3]^T$. The approximation $\tilde L$ of $L$ is similarly realized with a Monte Carlo 
approximation with $M:=10^4$ uniformly random points. We use the same setting as in Section \ref{sec:exp_square} for the kernel, its parameters, the 
greedy algorithm and the termination criteria, with the only difference that the points are selected starting from a set of uniformly random points on the 
sphere.

Moreover, we compare the greedy points with sets of minimal energy points (see e.g. \cite{Hardin2004}). For each $n$, these sets $X_n:=\{x_i\}_{i=1}^n$ are 
defined as minima on $(\mathbb S^2)^n$ of the Riesz energy
\begin{align*}
E(X_n) := \sum_{i=1}^{n} \sum_{j=i+1}^{n} \left\|  x_i - x_j  \right\|^{-2}.
\end{align*}
We use here the precomputed points from \cite{SpherePts}, which are found by numerical 
minimization of this functional.

In this case the density is not constant and does not have compact support, and this is reflected in the fact that the points selected by the greedy algorithm 
are 
more spread over the entire $\Omega$ (see Figure \ref{fig:sphere_weights}). Nevertheless, the points are more concentrated in the area where the density $\nu$ 
is larger, and the weights are accordingly larger. Also in this experiment, a few weights of small magnitude are negative.

Again, the convergence of the worst-case error (see Figure \ref{fig:sphere_convergence}) shows that the greedy algorithm produces quadrature weights with a 
worst-case error decaying like $n^{-\tau/d}$, if one sets $d=2$ as the dimension of $\Omega$ as an embedded manifold in $\R^3$. This is in accordance with 
known error estimate for kernel interpolation on manifolds (see \cite{Wright2012} and Remark \ref{rem:manifold}). The same behavior is clearly observed with 
integration with the minimal energy points.

\begin{figure}[h!]
\begin{center}
\includegraphics[width=\textwidth]{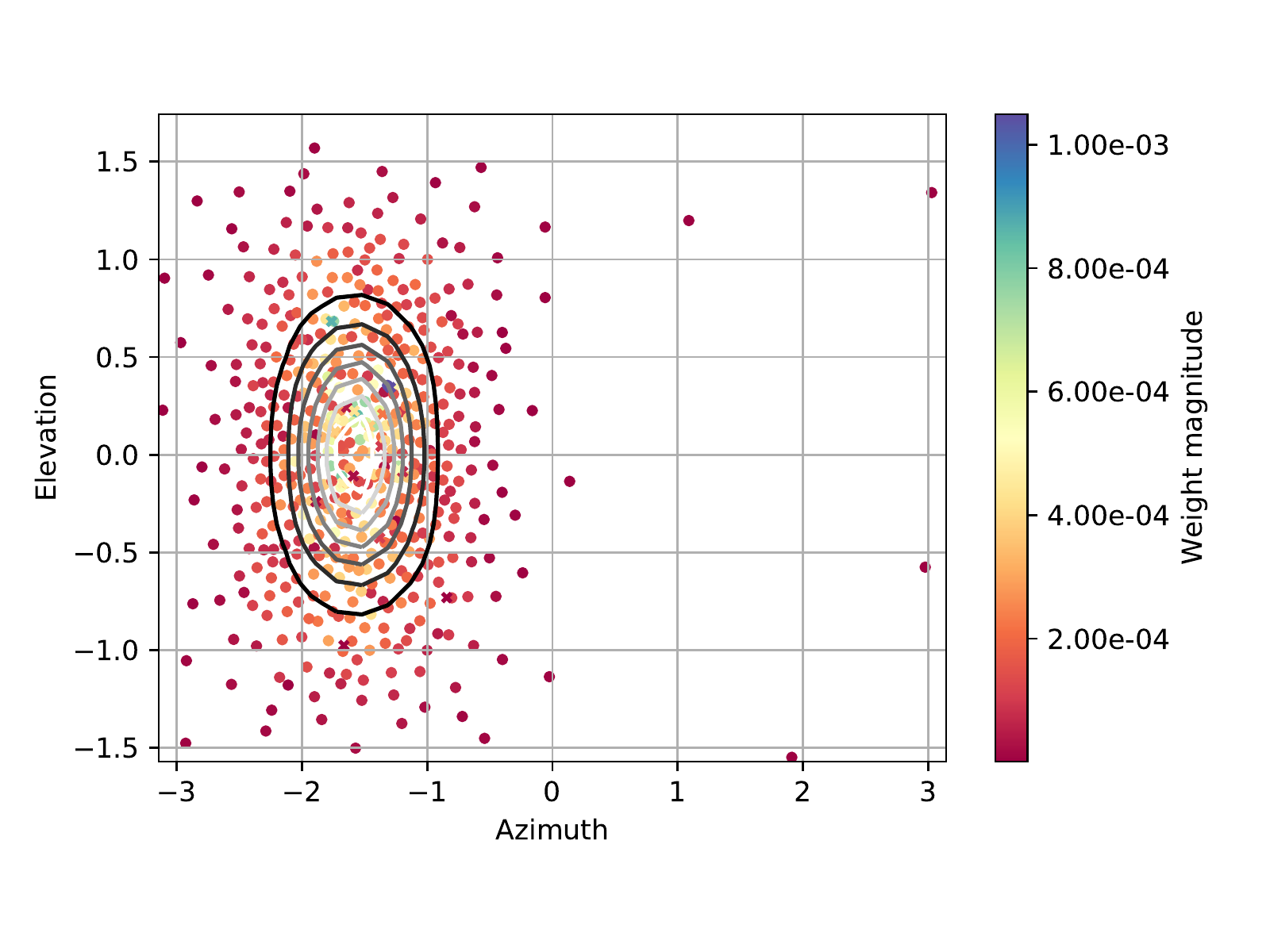}
\end{center}
\caption{Results of the greedy algorithm in the example of Section \ref{sec:exp_sphere}, representer in spherical coordinates. The figure shows the contour 
plot of the density $\nu$ (grayscale lines), and the position of the points selected by the algorithm, which are colored according to the magnitude of the 
corresponding positive (circles) or 
negative 
(crosses) weight.}
\label{fig:sphere_weights}
\end{figure}

\begin{figure}[h!]
\begin{center}
\includegraphics[width=\textwidth]{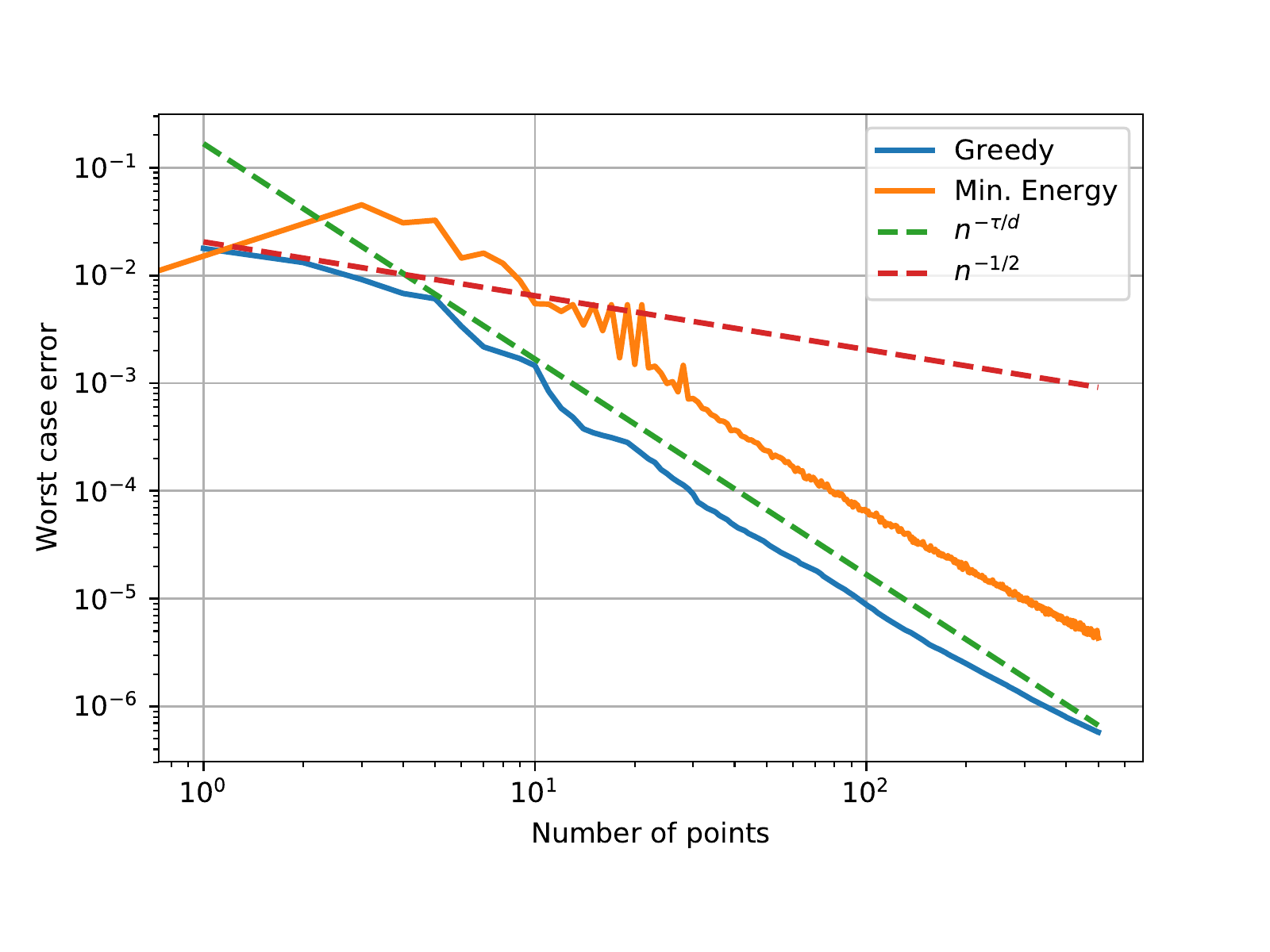}
\end{center}
\caption{
Decay of the worst-case error as a function of the number of points for the example of Section \ref{sec:exp_sphere}. The figure shows the error for the greedy 
points and for sets of minimal energy points, and rates of decay scaled to the greedy error. Observe that $d$ is the dimension of the manifold, and not of 
the embedding space. 
}
\label{fig:sphere_convergence}
\end{figure}

\subsection{An Uncertainty Quantification example}\label{sec:exp_uq}
As a final example we test the greedy algorithm on the benchmark Uncertainty Quantification (UQ) problem described in \cite{Koeppel2018}. 

We briefly describe the setting of the problem, and we refer to the cited paper for a thorough discussion. We 
have a Partial Differential Equation (PDE) modelling a two-phase flow in a porous medium, which depends on three input parameters (the injection rate, the 
relative permeability degree, and the reservoir 
porosity) and represents the saturation of some carbon dioxide which is injected into a one-dimensional aquifer over a time interval $[0, T]$. 

The aquifer is discretized into $250$ equal sized cells, and for a fixed value of the parameter triple $\theta:=[\theta_1, \theta_2, 
\theta_3]^T\in\R^3$, a time-dependent numerical solution of the PDE can be computed by the Finite Volume (FV) method. 
We denote as $s(\theta)\in\R^{250}$ the nodal values of the numerical solution at the final time $T$ computed with parameters $\theta$, and thus the FV 
discretization defines a map $\theta\in\R^3\mapsto s(\theta)\in\R^{250}$.

Instead than just computing the solution for a fixed value, in this case one is interested to quantify the effect on the solution $s(\theta)$ of an uncertain
knowledge of the input parameters. In particular, each parameter triple $\theta:=[\theta_1, \theta_2, \theta_3]^T$ is assumed to represent a sample from 
three random variables with known distributions, and thus the solution $s(\theta)$ itself is a random variable with values in $\R^{250}$. The goal of the 
benchmark problem is to estimate the mean $\mu_s\in\R^{250}$ and the standard deviation $\sigma_s\in\R^{250}$ of this random vector using as few solutions of 
the PDE as possible.

The benchmark includes also a dataset \cite{UQ_comparison_dataset} which contains a spatial discretization of the input parameters into $N:=10000$ 
points, i.e., a set $X:=\{\theta_i\}_{i=1}^N\subset\R^{3}$ representing independent samples of the parameters drawn accordingly to the respective 
distributions, 
and an implementation of the FV solver to compute the values $Y:=\{y_i:=s(\theta_i)\}_{i=1}^N\subset\R^{250}$. Moreover, the mean and standard deviation 
vectors computed by the various methods analyzed in \cite{Koeppel2018} are also available for comparison. As a reference solution, the paper uses the integral 
computed by a Monte Carlo approximation which uses the full set of nodes $X$.

In this case, we run the greedy algorithm with the same Matern kernel and $\gamma:=1/2$. Following \cite{Koeppel2018}, the points are selected from $X$ itself, 
since this discretization incorporates information of the distribution of the parameters and this is assumed to be a known information. The resulting 
quadrature rule is applied to each of the $250$ entries of the solution vector. Namely, if we assume that the points selected by the greedy algorithm are the 
subset $X_n:=\{\theta_{i_1}, \dots, \theta_{i_n}\}\subset X$, we obtain an approximated mean vector $\tilde \mu_s\in\R^{250}$ as
\begin{align*}
\left(\tilde \mu_s\right)_j := Q_{X_n}(s_j) = \sum_{k=1}^n w^*_k (y_{i_k})_j,
\end{align*}
where $(y_{i_k})_j$ is the $j$-th component of the $i_k$-th output vector. Similarly, the approximated standard deviation is the vector $\tilde 
\sigma_s\in\R^{250}$ with
\begin{align*}
\left(\tilde \sigma_s\right)_j := \left(   \sum_{k=1}^n w^*_k (y_{i_k})_j^2 -  \left(\tilde \mu_s\right)_j^2\right)^{1/2},
\end{align*}
where we used the fact that the variance is the difference between the mean of the square and the square of the mean.

Observe that this process is actually extracting a compressed quadrature rule from the reference Monte Carlo one.

For these approximated vectors, we can compute the $\ell_2$ errors with respect to the reference mean and standard deviation provided by the Monte Carlo 
quadrature. Following \cite{Koeppel2018}, these are computed as
\begin{align*}
E_\mu:= \frac{1}{250}\left\|\mu_s - \tilde\mu_s \right\|,\quad E_\sigma:= \frac{1}{250}\left\|\sigma_s - \tilde\sigma_s \right\|.
\end{align*}
The values of $E_{\mu}$, $E_\sigma$ are reported in Figure \ref{fig:UQ} for increasing values of the number of quadrature points. We also report the same 
errors obtained in the cited paper using some state-of-the-art methods, namely arbitrary polynomial chaos expansion (aPC), spatially adaptive sparse grids 
(aSG), $P$-greedy kernel interpolation (PGreedy), and Hybrid stochastic Galerkin (HSG). It is clear that the present method yields comparable results, and 
it even outperforms them if sufficiently many centers are used. 

We remark that in the plot (and in the paper) for each of the methods two parameter sets are 
tested. Also in the case of the algorithm of this paper, a quite high sensitivity to the parameter $\gamma$ of the kernel was observed, even if we report only 
the results for a representative value.

\begin{figure}[h!t]
\begin{center}
\begin{tabular}{c}
\includegraphics[width=\textwidth]{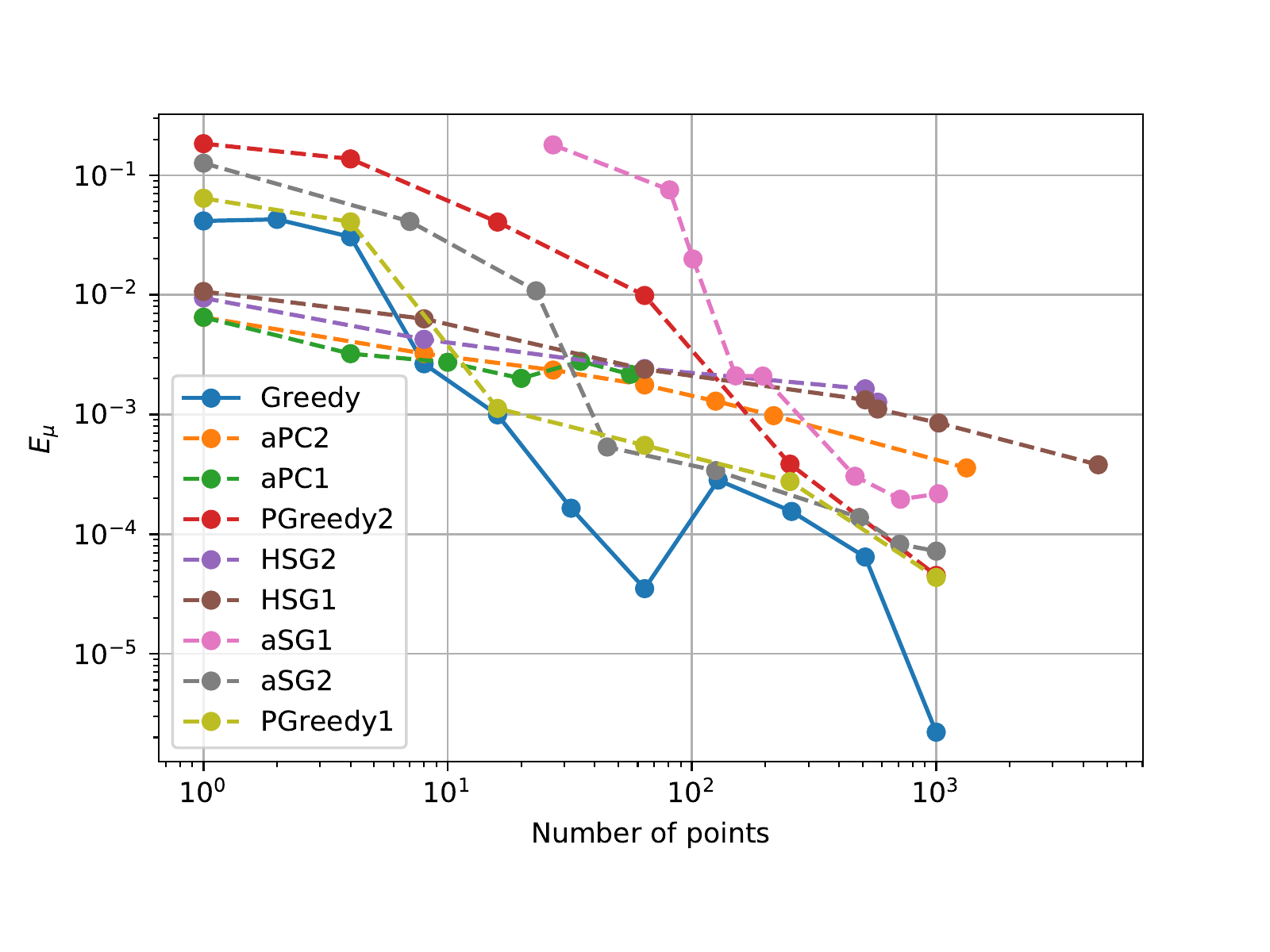}\\
\includegraphics[width=\textwidth]{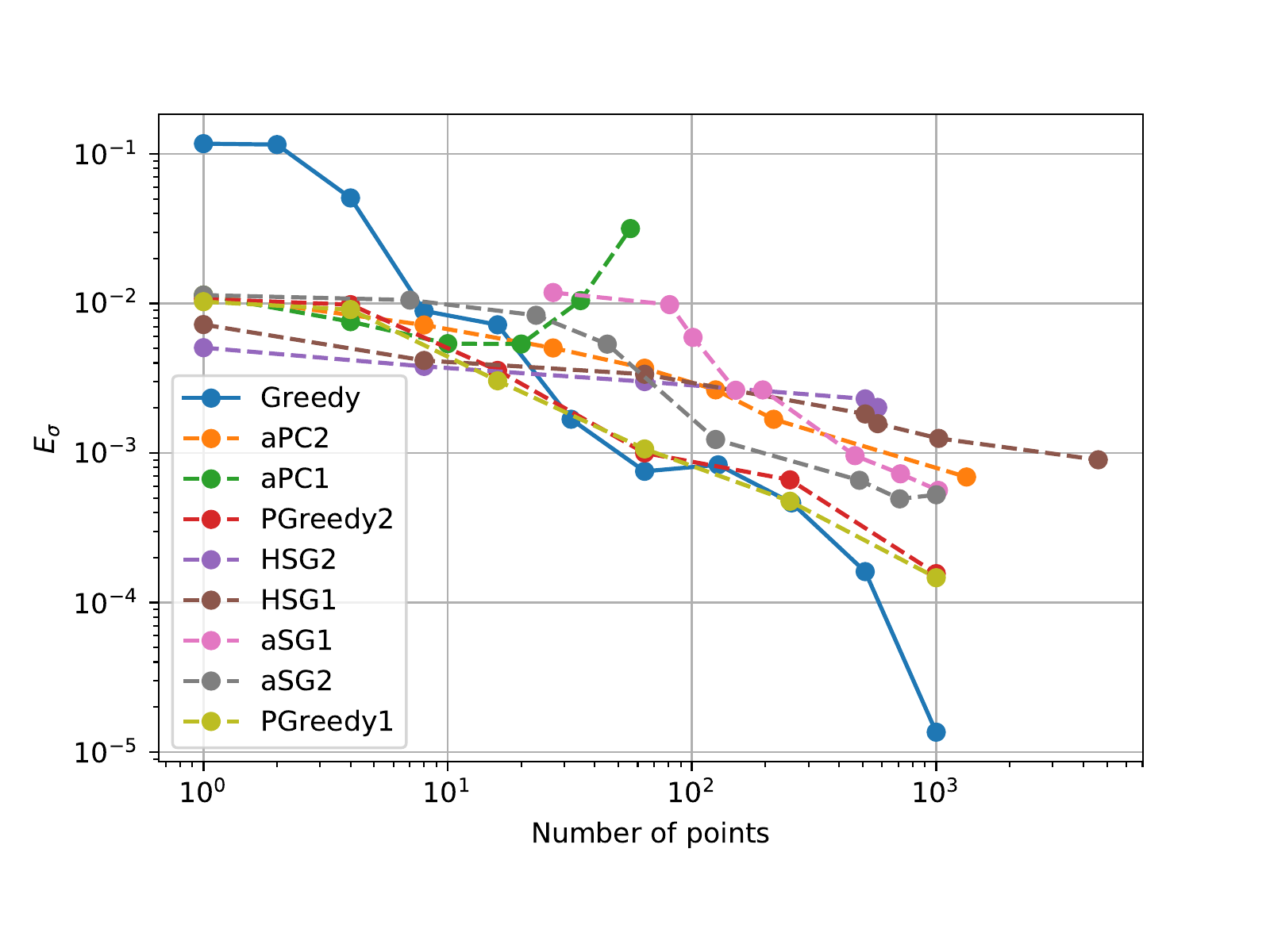}
\end{tabular}
\end{center}
\caption{Errors $E_{\mu}$ (approximation of the mean, upper figure) and $E_{\sigma}$ (approximation of the standard deviation, lower figure) as functions of 
the number of points obtained by the greedy algorithm in the example of Section \ref{sec:exp_uq}.
The results are compared with the ones obtained with 
arbitrary polynomial chaos expansion (aPC), spatially adaptive sparse grids 
(aSG), $P$-greedy kernel interpolation (PGreedy), and Hybrid stochastic Galerkin (HSG) from \cite{Koeppel2018}.
}\label{fig:UQ}
\end{figure}

\section{Future work}
The numerical experiments strongly suggest that the theoretical rates obtained for the greedy points are suboptimal, and this behavior would deserve 
additional investigation. Moreover, future work should focus on studying the sign of the weights, and on a stability analysis of the greedy  quadrature 
formulas.

\vspace{1cm}
\noindent\textbf{Acknowledgements:} We thank Tizian Wenzel for several comments on an early version of this manuscript.

\bibliography{bib-toni}
\bibliographystyle{abbrv}

\end{document}